\documentclass[a4paper,10pt]{article}
\usepackage[utf8]{inputenc}
\usepackage{setspace}

 \usepackage[lmargin=3cm, rmargin=3cm]{geometry}
\usepackage{mathtools}
\usepackage{commath}
\usepackage{amsmath}
\usepackage{amsthm}
\usepackage{amssymb}
\usepackage{amsfonts}
\usepackage{graphicx}
\usepackage{color}
\usepackage{calc}
\usepackage{IEEEtrantools}
\usepackage[percent]{overpic}
\usepackage{xcolor}
\usepackage{float}
\usepackage{enumerate}
\usepackage{tikz-cd}
\usetikzlibrary{decorations.pathmorphing} \tikzset{snake it/.style={decorate, decoration=snake}}
\usepackage{lmodern}

\newcommand{\bracedincludegraphics}[2][]{%
  \sbox0{$\vcenter{\hbox{\includegraphics[#1]{#2}}}$}%
  \left\lbrace
    \vphantom{\copy0}
  \right.\kern-\nulldelimiterspace
  \underbrace{\vrule width0pt depth \dimexpr\dp0 + .3ex\relax\box0}}

\newtheorem{theorem}{Theorem}
\newtheorem{lemma}[theorem]{Lemma}
\newtheorem{proposition}[theorem]{Proposition}

\theoremstyle{definition}

\title{Point-to-line last passage percolation and the invariant measure of a system of reflecting Brownian motions}
\author{Will FitzGerald and Jon Warren}

\begin{document}

\maketitle

\begin{abstract}
This paper proves an equality in law between 
the invariant measure of a reflected system of 
Brownian motions and a vector of point-to-line last passage percolation
times
in a discrete random environment. 
A consequence describes the distribution of the all-time supremum of 
Dyson Brownian motion with drift. A finite temperature version relates
the point-to-line 
partition functions of two directed polymers, with an inverse-gamma 
and a Brownian environment, and generalises Dufresne's identity. Our proof introduces an interacting system of 
Brownian motions with an invariant measure given by a field of 
point-to-line log partition functions for the log-gamma polymer.
\end{abstract}

\paragraph{2010 Mathematics Subjects Classifications.} 60J65, 60B20, 60K35.

\section{Introduction}
In this paper we generalise to a random matrix setting the  classical  identity:
\begin{equation} \label{expmax}
 \sup_{t \geq 0} \bigl( B(t) - \mu t \bigr) \stackrel{d}{=} e(\mu)
\end{equation}
where $B$ is a Brownian motion, $\mu > 0$ a drift and $e(\mu)$ is a random variable which has the exponential distribution
with rate $2\mu$. In our generalisation, the Brownian motion  is replaced by the largest eigenvalue process
of 
a Brownian motion with drift on the space of Hermitian matrices 
(see Section \ref{nr_br})
and the single exponentially distributed random variable   is replaced 
by  a random variable constructed from a field of independent exponentially distributed random variables using the operations of summation and maximum. 
In fact this latter random variable is well known as a point-to-line last passage percolation time. 
\begin{theorem}
\label{sup_lpp}
 Let $(H(t):t\geq 0)$ be an $n \times n$ Hermitian Brownian motion,
 let $D$ be 
 an $n \times n$ diagonal matrix with
 entries $D_{jj} = \alpha_{j} > 0$ for each $j = 1, \ldots n$ and let 
 $\lambda_{\text{max}}(A)$ denote the largest eigenvalue of a matrix $A$. Then
 \begin{equation*}
  \sup_{t \geq 0} \lambda_{\max}(H(t) - t D) \stackrel{d}{=} 
  \max_{\pi \in \Pi_n^{\text{flat}}}
  \sum_{(i j) \in \pi} e_{ij} 
 \end{equation*}
where $e_{ij}$ are an independent collection of 
exponential random variables indexed by 
$\mathbb{N}^2 \cap \{(i, j) : i + j \leq n+1\}$
with rates $\alpha_i + \alpha_{n + 1 -j}$
and the maximum is taken over the set of all directed (up and right)
nearest-neighbour paths from $(1, 1)$ 
to the line $\{(i, j) : i + j = n+1\}$ which we denote by $\Pi_n^{\text{flat}}$.
\end{theorem}

This result gives a
connection between random matrix theory and the 
\emph{Kardar-Parisi-Zhang} (KPZ) universality class, a collection of models 
related to random interface growth including growth models,
directed polymers in a random environment and various interacting particle systems.
Connections of this form originated in the seminal 
work of Baik, Deift, Johansson \cite{baik_deift_johansson}
showing that the limiting distribution 
of the largest increasing subsequence in a random 
permutation is given by the 
Tracy-Widom GUE distribution. 
They have been extensively studied 
since then: for \emph{curved} initial data (in our context 
\emph{point-to-point} last passage percolation) in
\cite{baryshnikov, johansson_2, oconnell_yor, o_connell2002, prahofer_spohn, warren} where the 
\emph{Robinson-Schensted-Knuth} (RSK)
correspondence plays a key role and for \emph{flat} initial data (in our context 
\emph{point-to-line} last passage percolation) in \cite{baik_rains, 
bisi_zygouras, bfps, 
ferrari, remenik_nguyen, sasamoto}
where 
the relationships 
are more mysterious.

 There are two results which particularly 
relate to Theorem \ref{sup_lpp}. 
Baik and Rains \cite{baik_rains} used a symmetrised version of the RSK correspondence to prove 
an equality in law between the point-to-line last passage percolation  
time 
and the largest eigenvalue from the \emph{Laguerre orthogonal 
ensemble} (LOE), see section \ref{nr_br} for the definition;
while a more recent work by Nguyen and Remenik \cite{remenik_nguyen} used the 
approach of multiplicative 
functionals from \cite{mult_functionals} to prove an equality in law between the
supremum of non-intersecting Brownian bridges and 
the square root of the largest eigenvalue of LOE. 
In Section \ref{nr_br} we show these two results can be combined to 
establish Theorem 
\ref{sup_lpp} in the case of \emph{equal drifts}: $\alpha_1= \alpha_2= \ldots =\alpha_n$.




One  aspect of the links between random matrices and growth models in the KPZ class is a striking variational representation for
the largest eigenvalue of Hermitian 
Brownian motion. Specifically, consider 
a system of reflected 
Brownian motions, where each particle is reflected up from 
the particle below (see Section 
\ref{brownian_queues}) then the largest particle of this system
is equal in distribution, as a process,  to  the largest eigenvalue of a Hermitian Brownian motion, see  \cite{baryshnikov, gravner_tracy_widom, 
o_connell2002, warren}). This can be combined with a time reversal, as in \cite{five_author},
to show that the 
all-time supremum of the largest eigenvalue has the same distribution as the largest particle 
in a stationary system of reflecting Brownian motions but with an 
\emph{additional reflecting wall at the origin}. This is a generalisation of the classical  argument that deduces from  the identity  \eqref{expmax}  that the invariant 
measure of a reflected Brownian motion with negative drift 
is the exponential distribution. Thus we are motivated to study the invariant measure of this  system of reflecting Brownian motions with a wall
and unexpectedly we find that the entire invariant measure -- rather than just the marginal distribution of the top particle -- can be described by last passage percolation.

Let $\alpha_j > 0$ for each $j = 1, \ldots, n$ 
and let $(B_1^{(-\alpha_1)}, \ldots, B_n^{(-\alpha_n)})$ be 
independent Brownian motions with drifts $(-\alpha_1, \ldots, -\alpha_n)$. 
A system of reflected Brownian motions with a wall at the origin 
can be defined inductively using the Skorokhod construction,
\begin{align}
 Y_1(t) & = B_1^{(-\alpha_1)}(t)- \inf_{0 \leq s \leq t} B_1^{(-\alpha_1)}(s) 
 = 
 \sup_{0 \leq s \leq t} \bigl( B_1^{(-\alpha_1)}(t) - B_1^{(-\alpha_1)}(s) \bigr) \label{defnY_1} \\
Y_j(t) & = \sup_{0 \leq s \leq t} \bigl(B_j^{(-\alpha_j)}(t)-B_j^{(-\alpha_j)}(s)
+ Y_{j - 1}(s)\bigr)
\text{ for } j \geq 2. \label{defnY_2}
\end{align}

We will show in Section \ref{brownian_queues}
that the distribution of $Y(t) = (Y_1(t), \ldots, Y_n(t))$ converges to a unique
invariant measure 
and we denote a random variable with this law by
$(Y_1^*, \ldots, Y_n^*)$. This is 
equal in distribution to a vector of
point-to-line last passage percolation times where we allow the point from which 
the directed paths start to vary: let
$\Pi_n^{\text{flat}}(k, l)$
denote the set of all directed (up and right) 
nearest-neighbour paths from the point $(k, l)$ 
to the line $\{(i, j) : i + j = n+1\}$ and let 
\begin{equation}  \label{defnG}
 G(k, l) = \max_{\pi \in \Pi_n^{\text{flat}}(k, l)}
  \sum_{(i j) \in \pi} e_{ij}
\end{equation}
where $e_{ij}$ are independent
exponential random variables indexed by $\mathbb{N}^2 \cap \{(i, j) : i + j \leq n+1\}$
with 
rates $\alpha_i + \alpha_{n -j+1}$.

\begin{theorem}
\label{equality_law}
Let $(Y_1^*, \ldots, Y_n^*)$ be distributed according to the 
invariant measure of the system of reflected Brownian motions  
defined by \eqref{defnY_1}, \eqref{defnY_2}
and  let 
$(G(1, n), \ldots, G(1, 1))$ be the vector of
point-to-line last passage percolation times defined by \eqref{defnG}.
For any $n \geq 1$,
 \begin{equation*}
(Y_1^*, \ldots, Y_n^*) \stackrel{d}{=} (G(1,n), \ldots, G(1, 1)).
 \end{equation*}
\end{theorem}

We will prove Theorem \ref{equality_law} by finding 
transition densities for both systems of a similar form to those found 
for TASEP in 
Sch{\"u}tz \cite{schutz} and 
reflected Brownian motions in Warren \cite{warren}
and use these to calculate explicit densities 
for both vectors. Then Theorem \ref{sup_lpp}, with general drifts, follows from 
Theorem \ref{equality_law} by the time reversal argument discussed previously. 

Point-to-line last passage percolation is related to the \emph{totally asymmetric 
exclusion process} (TASEP) by interpreting last
passage times as the time at which a particle jumps.
The point-to-line geometry corresponds to a periodic initial condition 
for TASEP, where 
particles are initially
located at 
every even site of the negative integers. The joint distribution 
of particle positions at a fixed time is given by a Fredholm determinant 
in \cite{bfps, sasamoto} and
under a suitable limit the authors obtain the $\text{Airy}_1$ process.
Their techniques also provide Fredholm determinants more generally, for example
for the vector 
$(G(1, n), \ldots, G(n, n))$.  In TASEP and in the systems
of reflected Brownian motion
studied in \cite{weiss} the role of the flat 
geometry is played by a \emph{periodic initial condition}, whereas for the Brownian model
$(Y(t))_{t \geq 0}$ considered above this role is played by a \emph{reflecting 
wall} at the origin. This is a substantial difference: a natural path-valued
process to 
consider is the evolution as $n$ varies of the path of the 
top particle; in this setting the techniques 
used in \cite{bfps, weiss, sasamoto} are no longer applicable.
The path of the top particle is a candidate for a finite $n$ analogue 
of the $\text{Airy}_1$ process. 


Another motivation for this reflected system is provided by queueing theory:
reflected Brownian motions have been considered as a model for 
\emph{tandem queues} 
in heavy traffic and the invariant measures have been studied 
extensively both analytically and numerically
\cite{budhijara_lee, dieker2009, dupuis_williams,
glynn_whitt, harrison1987, oconnell_yor}.
It is known from \cite{harrison1987} that the invariant measure 
has an explicit product form when a skew symmetry condition 
for the angles of reflection holds 
and it is known from \cite{dieker2009}
that the invariant measure can be expressed as a sum of exponential random variables
if a weaker relation between the angles holds. 
In our case, 
the presence of a wall, which has a natural 
queueing interpretation 
as a deterministic arrival 
process, ensures that the skew symmetry condition fails;
nonetheless
Theorem \ref{equality_law} describes the non-reversible invariant measure
and we give 
an explicit formula for its density in Section \ref{brownian_queues}.

A further classical result from probability theory that we consider
is Dufresne's identity. Let $\mu > 0$, let 
$B^{(-\mu)}$ be a Brownian motion with drift $-\mu$ and 
let $\gamma^{-1}(\mu)$ denote an inverse gamma random variable with 
shape parameter $\mu$ and rate $1$. Then Dufresne's identity 
is an equality in law,
\begin{equation*}
 2 \int_0^{\infty} e^{2B^{(-\mu)}(t)} dt \stackrel{d}{=} \gamma^{-1}(\mu)
\end{equation*}
which has been studied in 
mathematical finance and diffusion in a random environment (see 
\cite{matsumoto_yor, yor} and the references within).
This is a positive temperature version of the fact that the 
all-time supremum of Brownian motion with negative drift 
has an exponential distribution and suggests the following 
positive  temperature version of Theorem \ref{sup_lpp}. 

\begin{theorem}
\label{finite_temp_thm}
For $i = 1, \ldots, n$ let $\alpha_i > 0$ and let $B_i^{(-\alpha_i)}$ be 
independent Brownian motions with drifts 
$-\alpha_i$. Let $W_{ij}$ be a collection of inverse gamma 
random variables indexed by $\mathbb{N}^2 \cap \{(i, j) : i + j \leq n+1\}$ with 
shape parameters $\alpha_i + 
\alpha_{n - j + 1}$ and rate $1$ and let $\Pi_n^{\text{flat}}$
denote the set of all directed paths from $(1, 1)$ 
to the line $\{(i, j) : i + j = n+1\}$. Then 
\begin{equation*}
\int_{0 =s_0< s_1 < \ldots < s_{n} < \infty} e^{\sum_{i=1}^n
B_i^{(-\alpha_i)}(s_i) - B_i^{(-\alpha_i)}(s_{i-1})}
 ds_1 \ldots ds_{n} \stackrel{d}{=}  2
   \sum_{\pi \in \Pi_n^{\text{flat}}}
  \prod_{(i j) \in \pi} W_{ij}. 
\end{equation*}
\end{theorem}
The left hand side of this expression is the partition function for a 
point-to-line polymer in a Brownian environment while
the right hand side is the partition function for the point-to-line 
log-gamma polymer.
The \emph{point-to-point} polymers have been 
studied in a number of recent papers: the Brownian model
in \cite{bcf, oconnell_yor, o_connell2012} and  
the log-gamma polymer in 
\cite{borodin_corwin_remenik, cosz, oconnell_sep_zygouras, sep} 
with one motivation being their relationship to the KPZ 
equation (see \cite{corwin_survey} for a survey). 
The \emph{point-to-line} log-gamma polymer, which corresponds to a flat 
initial condition for the KPZ equation, has been studied recently 
by \cite{bisi_zygouras, nguyen_zygouras} using a local version of 
the geometric RSK correspondence
and  an expression is given for 
the Laplace transform of the point-to-line partition function of
the log-gamma polymer
in terms of Whittaker functions.  From
Theorem \ref{finite_temp_thm} 
it follows that the Laplace transform of the partition function of the  point-to-line Brownian model,  which has not been studied 
previously, is also given by the same expression.

For the proof, we use a time reversal  argument to show
that Theorem \ref{finite_temp_thm} follows from a stronger 
result on the 
invariant measure of a system of Brownian motions where the
reflection rules of the system in Theorem \ref{equality_law} 
are replaced by smooth exponential interactions. We find this invariant measure 
by embedding the Brownian system in a larger system of interacting Brownian
motions, indexed by a triangular array, such that the invariant measure 
of this system is given by 
a field of point-to-line log partition functions for the log-gamma polymer.

\section{Equal drifts and connections to LOE}
\label{nr_br}
This section discusses in more detail 
the connection between the results of Nguyen and Remenik \cite{remenik_nguyen}, 
and Baik and Rains \cite{baik_rains}.

We first introduce the relevant random matrix ensembles and processes.
We consider a Brownian motion on the space of
$n \times n$ Hermitian matrices denoted 
$(H(t))_{t \geq 0}$ and constructed from independent entries
$\{H_{i, j} : i \leq j\}$ such that along the diagonal 
$H_{ii}$ are real  standard Brownian motions, the 
entries below the diagonal $\{H_{ij} : i < j\}$ are standard
complex Brownian motions, 
and the remaining entries are determined by the Hermitian constraint 
$H_{ij} = \bar{H}_{ji}$. 
The ordered eigenvalues $\lambda_1, \ldots, \lambda_n$ form a system of Brownian 
motions conditioned 
(in the sense of Doob) not to collide and with a specified entrance law
from the origin which can be constructed as a limit 
from the interior of the Weyl chamber (for example, see \cite{konig2005}). 
The time changed matrix-valued 
process 
$(H^{\text{br}}(t))_{t \geq 0} = ((1-t) H(t/(1-t)))_{t \geq 0}$ 
is a Brownian bridge in the
space of Hermitian matrices and the eigenvalues are given by 
applying this time change to the above system of Brownian motions conditioned not to 
collide. It can be checked, for example
by calculating the joint distribution of particles at a sequence of times, 
that the eigenvalues of a Hermitian Brownian bridge are given by 
a system of Brownian bridges 
which we denote $(B_1^{\text{br}}, \ldots, B_n^{\text{br}})$
all started at zero at time 0 and
ending at zero at time $1$ with a specified entrance and exit law constructed 
as a limit from the interior of the Weyl chamber, and conditioned 
(in the sense of Doob) not to collide in the time interval $t \in (0, 1)$.

Let $X$ be an $m \times n$ matrix with entries given by 
independent standard normal random variables and assume $m \geq n$. 
Then $M = X^T X$ is an $n \times n$ matrix from the \emph{Laguerre orthogonal 
ensemble} (LOE) and 
the joint density of eigenvalues is given by 
\begin{equation*}
f_{\text{LOE}}(\lambda_1, \ldots, \lambda_n) = \frac{1}{c_n}
\prod_{1 \leq i < j \leq n} \lvert \lambda_i - \lambda_j \rvert 
 \prod_{i = 1}^n \lambda^a_i e^{-\lambda_i/2}
\end{equation*}
where $c_n$ is a normalisation constant and the parameter $a = (m - n - 1)/2$. 
Throughout this paper we
will only be interested in the case $a = 0$, or equivalently $m = n + 1$. 
The main result of Nguyen and Remenik \cite{remenik_nguyen} states that
\begin{equation*}
 4\left(\sup_{0 \leq t \leq 1} B_n^{\text{br}}(t)\right)^2  \stackrel{d}{=} 
 \lambda_{\max}^{\text{LOE}}.
\end{equation*}
We use the time change between Hermitian Brownian motions and 
bridges to express this in terms of a Hermitian Brownian motion:
\begin{IEEEeqnarray*}{rCl}
 P(B_n^{\text{br}}(t) \leq x \text{ for all }t \in [0, 1]) & = & 
 P((1-t) \lambda_{\max}(H(t/1-t)) \leq x \text{ for all } t \geq 0) \\
 & = & P(\lambda_{\max}(H(u)) \leq x(1+u) \text{ for all } u \geq 0) \\
 & = & P(x \lambda_{\max}(H(v/x^2)) \leq x^2 + v \text{ for all } v \geq 0) \\
 & = & P(\sup_{t \geq 0} \lambda_{\max}(H(t) - tI) \leq x^2)
\end{IEEEeqnarray*}
where the change of variables are given by $u = t/(1-t)$ and $v = u x^2$ 
and the largest eigenvalue inherits 
the scaling property of Brownian motion.
Therefore 
\begin{equation}
\label{time_change_loe}
4 \sup_{t \geq 0} \lambda_{\max}(H(t) - tI) \stackrel{d}{=} 
 \lambda_{\max}^{\text{LOE}}.
\end{equation}

This is connected to last passage percolation by 
the results of Baik and Rains \cite{baik_rains}. We
refer to Section 10.5 and 10.8.2 of Forrester \cite{forrester_book}
for the precise statements we use which are obtained after taking a 
suitable limit of the geometric data considered in \cite{baik_rains} 
to exponential data. 
Let
$\Pi_n^{\text{flat}}$
denote the set of all directed nearest-neighbour 
paths from the point $(1, 1)$ 
to the line $\{(i, j) : i + j = n+1\},$ where the directed 
paths consist only of up and right steps: that is  to say, paths whose co-ordinates are non-decreasing.
We let $e_{ij}$ be independent exponential random variables indexed by $\mathbb{N}^2 \cap \{(i, j) : i + j \leq n+1\}$
with 
rate $\alpha_i + \alpha_{n - j+1}$ and
define the last passage percolation time
\begin{equation*}
 G(1, 1) = \max_{\pi \in \Pi_n^{\text{flat}}(k, l)}
  \sum_{(i j) \in \pi} e_{ij}.
\end{equation*}
This can be compared with \emph{point-to-point}
last passage percolation in a \emph{symmetric random environment}.  
Fix $n$ and define exponential data
$\{\hat{e}_{ij} : i, j \leq n \}$ by 
$ \hat{e}_{ij} = \hat{e}_{ji} = e_{ij} $ for $i < n - j + 1$, 
and $\hat{e}_{ij} = \frac{1}{2} e_{ij}$ for $i = n - j + 1$.         
%
Let $\Pi_n$ denote the set of all directed (up and right) nearest-neighbour 
paths from the point $(1, 1)$ 
to the point $(n, n)$. 
Due to the symmetry of the random environment
\begin{equation}
\label{symmetry}
 2 \max_{\pi \in \Pi_n^{\text{flat}}} \sum_{(ij) \in \pi} e_{ij} = 
 \max_{\pi \in \Pi_n} \sum_{(ij) \in \pi} \hat{e}_{ij}.
\end{equation}
The RSK correspondence can be applied to \emph{any rectangular array 
of data} and generates a pair of semi-standard Young tableaux $(P, Q)$ with shape 
$\nu$ such that $\nu_1$ is 
equal to the point-to-point 
last passage percolation time. When applied to 
exponential data with symmetry (see Section 
10.5.1 of Forrester \cite{forrester_book}), the two tableaux can 
be constructed from each other and 
the distribution of $\nu$
has a density with respect to Lebesgue measure 
given by
\begin{equation*}
f_{\text{RSK}}(x_1, \ldots, x_n) = \frac{\prod_{i = 1}^n \alpha_i \prod_{i < j} 
(\alpha_i+\alpha_j)}{\prod_{i < j} (\alpha_i - \alpha_j)} 
\text{det}(e^{-\alpha_i x_j})
\end{equation*}
for distinct $\alpha$. 
In the case when $\alpha_i = 1$ for each $i = 1, \ldots, n$ this can 
be evaluated as a limit and gives the eigenvalue density for
LOE (scaled by a constant factor of 2).
In combination 
with equation (\ref{symmetry}) this shows that,
\begin{equation}
\label{rsk_lpp}
 4 \max_{\pi \in \Pi_n^{\text{flat}}} \sum_{(ij) \in \pi} e_{ij} 
  \stackrel{d}{=} 
 \lambda_{\max}^{\text{LOE}}
\end{equation}
Therefore the combination of equation (\ref{time_change_loe})
and (\ref{rsk_lpp}) 
proves
Theorem \ref{sup_lpp} in the case when $D$ is a \emph{multiple 
of the identity matrix}.
We could  use this time change argument in the reverse direction to provide an alternative proof of 
Nguyen and Remenik starting from equation (\ref{rsk_lpp})
and our proof of Theorem \ref{sup_lpp}.

\section{Reflected Brownian motions with a wall}
\label{brownian_queues}

\subsection{Time reversal}

In the introduction we defined a system of reflected Brownian motions 
with a \emph{wall at the origin} $Y = (Y_1, \ldots, Y_n)$
and we now define the system \emph{without 
the wall}. 
Let $\alpha_j > 0$ for each $j = 1, \ldots, n$ 
and let $(B_1^{(-\alpha_n)}, \ldots, B_n^{-(\alpha_1)})$ be 
independent Brownian motions with drifts. 
A system of reflected Brownian motions
can be defined inductively using the Skorokhod construction,
\begin{align*}
 Z_1^n(t) & = B_1^{(-\alpha_n)}(t) 
\\
Z_j^{n}(t) & = \sup_{0 \leq s \leq t} (B_j^{(-\alpha_{n-j+1})}(t)-
B_j^{(-\alpha_{n-j+1})}(s)
+ Z_{j - 1}^{n}(s))
\text{ for } j \geq 2.
\end{align*}
An iterative application of the above gives 
the $n$-th particle the representation
\begin{equation}
\label{Z_repr}
 Z_n^n(t) = \sup_{0 = t_0 \leq t_1 \leq \ldots t_n = t} \sum_{i = 1}^n 
 (B_i^{(-\alpha_{n-i+1})}(t_i) - B_i^{(-\alpha_{n-i+1})}(t_{i-1})).
\end{equation}
This gives an interpretation of the largest particle in a 
reflected system as a point-to-point last 
passage percolation time 
in a Brownian environment. 
Similarly  the $n$-th particle in the system with a wall  defined by (\ref{defnY_1}, \ref{defnY_2}) has a 
representation
\begin{equation}
\label{Y_repr}
 Y_n(t) = \sup_{0 \leq t_0 \leq \ldots \leq t_n = t} \sum_{i = 1}^n 
 (B_i^{(-\alpha_i)}(t_{i}) - B_i^{(-\alpha_i)}(t_{i-1})),
\end{equation}
where the only difference is that there is one extra supremum over $t_0$
and we have reversed the order of the drifts. 
These systems are related: in 
\cite{five_author} it was proved in the zero drift case that for each 
fixed $t$,
\begin{equation*}
Y_n(t) \stackrel{d}{=} \sup_{0 \leq s \leq t} Z_n^n(s) 
\end{equation*}
by a time reversal argument which easily extends to the case with drifts. 
We prove a vectorised version of this time reversal which 
can also be useful for studying the full vector 
$(Y_1, \ldots, Y_n)$.
We first extend the definition of $Z$ to a triangular array $Z = (Z_j^k : 1 \leq j \leq k, 
1 \leq k \leq n)$ as follows
\begin{align}
\label{defnZ}
 Z_1^k(t) & = B_{n-k+1}^{(-\alpha_k)}(t) \text{ for } 1 \leq k \leq n\\
 Z_j^k(t) & = \sup_{0 \leq s \leq t} B_{n-k+j}^{(-\alpha_{k-j+1})}(t)
 - B_{n-k+j}^{(-\alpha_{k-j+1})}(s) + Z_{j-1}^k(s) \text{ for } 2 \leq j \leq k
\end{align}
with the representation
\begin{equation*}
 Z_j^k(t) = \sup_{0 = t_0 \leq \ldots \leq t_j = t} \sum_{i=1}^j 
 (B_{n-k+i}^{(-\alpha_{k-i+1})}(t_i) -
 B_{n-k+i}^{(-\alpha_{k-i+1})}(t_{i-1})). 
\end{equation*}
We note that the $Z$ process is still constructed from 
only $n$ independent Brownian motions.

\begin{proposition}
\label{time_reversal_vector}
For any fixed $t$, let $(Y_1, \ldots, Y_n)$ be defined by equation (\ref{defnY_1}, \ref{defnY_2}) 
and $(Z_1^1, Z^2_2, \ldots, Z_n^n)$ by equation (\ref{defnZ}), then, for any fixed $t\geq 0$,
\begin{equation*}
 (Y_1(t), \ldots, Y_n(t)) \stackrel{d}{=} 
 \left(\sup_{0 \leq s \leq t} Z_1^1(s), \ldots, \sup_{0 \leq s \leq t}
 Z_n^n(s)\right).
\end{equation*}
\end{proposition}
In particular, the equality in law of the marginal distribution of the last 
co-ordinate gives the extension of \cite{five_author} to general drifts,
\begin{equation*}
 Y_n(t) \stackrel{d}{=} \sup_{0 \leq s \leq t} Z_n^n(s). 
\end{equation*}

\begin{proof}
Fix 
$t$ and observe that
\begin{IEEEeqnarray*}{rCl}
 (Y_k(t))_{k=1}^n & = & 
 \left(\sup_{0 \leq t_0 \leq \ldots t_k = t} \sum_{i = 1}^k 
 (B_i^{(-\alpha_i)}(t_{i}) - B_i^{(-\alpha_i)}(t_{i-1})) \right)_{k=1}^n\\
  & = & \left(\sup_{0 = u_0 \leq \ldots u_k \leq t} \sum_{i = 1}^k 
 (B_{i}^{(-\alpha_{i})}(t - u_{k-i}) - B_{i}^{(-\alpha_{i})}(t
 - u_{k-i+1}))\right)_{k=1}^n
\end{IEEEeqnarray*}
by letting $t - u_i = t_{k-i}$.
By time reversal $(B_{i}^{(-\alpha_{i})}(t) - B_{i}^{(-\alpha_{i})}
(t - s))_{s \geq 0} \stackrel{d}{=} (B_{n-i+1}^{(-\alpha_{i})}(s))_{s \geq 0}.$  
Therefore
\begin{IEEEeqnarray*}{rCl}
 (Y_k(t))_{k=1}^n & \stackrel{d}{=} & 
 \left( \sup_{0 = u_0 \leq \ldots u_k \leq t} 
 \sum_{i = 1}^k 
 (B_{n-i+1}^{(-\alpha_{i})}(u_{k-i+1}) - 
 B_{n-i+1}^{(-\alpha_{i})}(u_{k-i})) 
 \right)_{k=1}^n   \\
 & = & \left(\sup_{0 \leq s \leq t} Z^k_k(s)\right)_{k = 1}^n 
\end{IEEEeqnarray*} 
where the final equality requires changing the index of summation from 
$i$ to $k-i+1$.
\end{proof}


\begin{proposition}
\label{time_change_refl}
 For $i = 1, \ldots, n$, let $\alpha_i > 0$.
  \begin{enumerate}[(i)]
\item The vector
$\left(\sup_{0 \leq s \leq t} Z_1^1(s), \ldots, \sup_{0 \leq s \leq t}
 Z_n^n(s)\right)$
converges almost surely as $t \rightarrow \infty$
to a finite random variable.
From this and Proposition \ref{time_reversal_vector} we can deduce that $(Y_1(t), \ldots, Y_n(t))$ converges in distribution  as $t \rightarrow \infty$
to a random variable which we denote $(Y_1^*, \ldots, Y_n^*)$ and satisfies
\begin{equation*}
 (Y_1^*, \ldots, Y_n^*) \stackrel{d}{=} \left(\sup_{0 \leq s \leq \infty} Z_1^1(s), \ldots, \sup_{0 \leq s \leq \infty}
 Z_n^n(s)\right).
\end{equation*}
%
 \item The top particle satisfies
 \begin{equation*}
  Y_n^* \stackrel{d}{=} \sup_{0 \leq t < \infty} Z_n^n(t)
  \stackrel{d}{=} \sup_{0 \leq t <\infty} \lambda_{\max}(H(t) - tD). 
 \end{equation*}
 \item Suppose that $\alpha_i = 1$ for all $i = 1, \ldots, n,$ then
 the top  particle satisfies
\begin{equation*}
   4Y_n^* \stackrel{d}{=}  4\sup_{0 \leq t < \infty} Z_n^n(t) \stackrel{d}{=}
  4\sup_{0 \leq t < \infty} \lambda_{\max}(H(t) - tI)
  \stackrel{d}{=}  \lambda_{\max}^{\text{LOE}} .
\end{equation*}
 \end{enumerate}
\end{proposition}

The random variable $(Y_1^*, \ldots, Y_n^*)$ 
is distributed according to the unique invariant measure of the Markov process $Y$, which
 will follow from Lemma \ref{exist_unique}.

\begin{proof}
We first show the almost sure convergence in part (i). 

It is sufficient to  show the suprema $\left(\sup_{0 \leq s \leq \infty} Z_1^1(s), \ldots, \sup_{0 \leq s \leq \infty}
 Z_n^n(s)\right)$ are almost surely finite. 
We prove a  stronger statement that will be useful later, namely, that 
\[
\lim_{t \rightarrow \infty} \frac{1}{t}Z^k_j(t) = -\min( \alpha_k, \alpha_{k-1}, \ldots, \alpha_{k-j+1}).
\]
Denote $\min( \alpha_k, \alpha_{k-1}, \ldots, \alpha_{k-j+1})$ by $\delta^k_j$. We proceed, for each $k$, by induction  on $j$. 

For $j = 1$, we have 
$Z_1^k(t) = B_{n-k+1}^{(-\alpha_k)}(t)$ and the required statement is a property of  Brownian motion with drift.
For the inductive step, 
\begin{IEEEeqnarray*}{rCl}
 Z^{k}_j(t) & = & \sup_{0 \leq s \leq t} 
 B_{n-k+j}^{(-\alpha_{k-j+1})}(t) - B_{n-k+j}^{(-\alpha_{k-j+1})}(s) + Z^{k}_{j-1}(s) \\
& = &  
 B_{n-k+j}^{(-\alpha_{k-j+1})}(t) + \sup_{0 \leq s \leq t} \bigl(- B_{n-k+j}^{(-\alpha_{k-j+1})}(s) + Z^k_{j-1}(s) \bigr).
\end{IEEEeqnarray*}
Now observe that  $B_{n-k+j}^{(-\alpha_{k-j+1})}(t)/t \rightarrow -\alpha_{k-j+1}$, and,  making use of the inductive hypothesis,
\[
\frac{1}{t} \sup_{0 \leq s \leq t}\bigl(- B_{n-k+j}^{(-\alpha_{k-j+1})}(s) + Z^k_{j-1}(s) \bigr) \rightarrow  \max(0, \alpha_{k-j+1}-\delta^k_{j-1}).
\]
Thus we deduce that  $ Z^{k}_j(t)/t$ 
tends to $-\min( \alpha_{k-j+1}, \delta^k_{j-1})= \delta^k_j$.


For parts (ii) and (iii), the first equality in 
distribution follows by the time reversal at the start of this section. 
The second equality in 
distribution follows from the well known equality in distribution of 
processes between the largest particle in a reflected system
of Brownian motions and  
the largest eigenvalue of Hermitian Brownian motion. For equal parameters
a proof can be found in any of
\cite{baryshnikov, gravner_tracy_widom, 
o_connell2002, warren} and for general parameters 
a proof can be found in \cite{interlacing_diffusions}. 
The final equality in distribution for part (iii) follows from the results of 
Nguyen and Remenik and the time change in Section \ref{nr_br}.
\end{proof}

The fluctuations of the largest eigenvalue of the 
Laguerre orthogonal ensemble are governed in the large $n$ limit by
the Tracy-Widom GOE distribution. This distribution arises as the scaling limit 
for models in the KPZ universality class with flat initial data and so we now see that  (the marginals of)  the stationary distribution  of reflecting Brownian motions with a wall also lies within this universality class. This is explained by equation 
(\ref{Y_repr})
or the relationship to $\sup_{0 \leq s \leq \infty} Z_n(s)$ along with 
equation (\ref{Z_repr}) which both
identify $Y_n^*$ 
as a point-to-line last passage percolation time in a 
Brownian environment.

\subsection{Transition Density}
The system of reflected Brownian motions with a wall can be defined through 
a system of SDEs
and we use this to define the process with a general initial condition.
Let $0 \leq y_1 \leq y_2 \leq \ldots \leq y_n$ and define
\begin{equation}
\label{sdes}
 Y_j(t) = y_j + B^{(-\alpha_j)}_j(t) + L_j(t) \text{ for } j = 1, \ldots, n 
\end{equation}
where $L_1$ is the local time process at zero of $Y_1$ and $L_j$ is the 
local time process at zero of $Y_j - Y_{j-1}$ for each $j = 2, 
\ldots, n$. This is a Markov process and we give its transition density. This has a form similar to
\cite{interlacing_diffusions, five_author, schutz, warren, weiss}.
Let $W_n^+ = \{0  \leq z_1 \leq \ldots \leq z_n\}$ denote the state space of
a system of reflected Brownian motions with a wall.
We define differential and integral operators acting on 
infinitely differentiable functions 
$f : [0, \infty) \rightarrow \mathbb{R}$ which have superexponential decay at infinity 
as follows,
\begin{equation}
\label{J_D}
D^{\beta} f(x) = f'(x) - \beta f(x),  \qquad 
J^{\beta} f(x)= \int_x^{\infty} e^{\beta(x - t)} f(t) dt
\end{equation}
where we define the derivative at zero to be the right derivative at zero.
The operators satisfy easy to verify identities:
\begin{enumerate}[(i)]
 \item Commutation relations: for any real $\alpha, \beta$,
 \begin{equation*}
  J^{\beta} D^{\alpha} = D^{\alpha} J^{\beta}, \qquad 
  J^{\beta} J^{\alpha} = J^{\alpha} J^{\beta},  
  \qquad D^{\beta} D^{\alpha} = D^{\alpha} D^{\beta}
 \end{equation*}
\item Inverse relations: let $\text{Id}$ denote the identity map, 
for any real $\alpha$,
\begin{equation*}
 D^{\alpha} J^{\alpha} = -\text{Id}, \qquad  J^{\alpha} D^{\alpha} = -\text{Id}
\end{equation*}
\item Relations to ordinary differentiation and integration: 
for any real $\alpha$,
\begin{equation*}
 D^{\alpha}f(x) = e^{\alpha x} D^0 (e^{-\alpha x} f(x) )
 \qquad J^{\alpha}f(x) = e^{\alpha x} J^0 (e^{-\alpha x} f(x)).
\end{equation*}
\end{enumerate}
We use the notation $D^{\alpha_1, \ldots, \alpha_n} = D^{\alpha_1} \ldots D^{\alpha_n}$
and $J^{\alpha_1, \ldots, \alpha_n} = J^{\alpha_1} \ldots J^{\alpha_n}$
to denote concatenated operations
and $D_x^{\alpha}, J_x^{\alpha}$ in order to specify a variable $x$ on which 
the operators act. 
We note that when the operators act on different variables they also commute.
Let $\phi_t^{(\alpha)}$ (resp. $\psi_t^{(\alpha)}$ and $\eta_t^{(\alpha)}$) be the
transition density of a Brownian motion (resp. Brownian motion killed at 
the origin and reflected at the origin) with drift $\alpha$. 
When the drift is zero we may omit the superscript.
Observe that $\psi_t(x, y) = \phi_t(y - x) - \phi_t(y + x)$ for all $x, y \geq 0$.
The right hand side can be defined for all $x, y$ and can be used 
to specify the right derivative of $\psi_t$ at zero to ensure that
the operation $D$ can be applied to $\psi_t$. 
A similar procedure can be used to specify the right derivative at zero 
of $\psi_t^{(\alpha)}, 
\eta_t^{(\alpha)}$ and all of these functions lie in the class of functions specified at the start of this section.  
We define 
\begin{equation*}
 r_t(x, y) = e^{-\sum_{i=1}^n \alpha_i(y_i - x_i) - \alpha_i^2t/2}
 \text{det}(D_{y_j}^{\alpha_1 \ldots \alpha_j} J_{x_i}^{-\alpha_1 \ldots -\alpha_i}
 \psi_t (x_i, y_j))_{i, j = 1}^n.
\end{equation*}

\begin{proposition}
\label{transition_density_Brownian}
The transition probabilities of $(Y_1(t), \ldots, Y_n(t))_{t \geq 0}$ 
have a density with respect to Lebesgue measure given by $r_t(x, y)$.
\end{proposition}

%
%

The following calculation shows that the proposition holds in the case $n = 1$ by using Siegmund 
duality. This can be stated in an integral form, for any fixed $t$,
\begin{equation*}
 \int_0^y \eta_t^{(-\alpha)}(x, u) du = \int_x^{\infty} \psi_t^{(\alpha)}(y, v) dv. 
\end{equation*}
We differentiate this expression in $y$, apply Girsanov's theorem
and symmetry to the 
killed Brownian motion
and use the identities in (iii) to obtain for all $x, y \geq 0$,
\begin{IEEEeqnarray*}{rCl}
  \eta_t^{(-\alpha)}(x, y) =  
  D^0_y J_x^0 \psi_t^{(\alpha)}(y, x)
 & = & 
 D^0_y J^0_x e^{-\alpha(y - x) - 
 \alpha^2 t/2} \psi_t(y, x) \\
& = & 
 e^{-\alpha(y - x) - 
 \alpha^2 t/2} D^\alpha_y J^{-\alpha}_x \psi_t(x, y).
\end{IEEEeqnarray*}
%
In the case of equal drifts this identity can be used to give an alternative form
of Proposition 
\ref{transition_density_Brownian}.
For $k \geq 1$ let $J^{(k)}$ (resp. $(D^{(k)}$)
denote $J^0$ (resp. $D^0$) concatenated $k$ times. 
%
Define  \begin{equation*}
  \bar{r}_t(x, y) = \text{det}(D_{y_j}^{(j - 1)} J_{x_i}^{(-i + 1)}
  \eta_t^{(-1)}(x_i, y_j))_{i, j = 1}^n.
 \end{equation*}
The transition probabilities of 
$(Y_1(t), \ldots, Y_n(t))_{t \geq 0}$ with drift vector 
$(-1, \ldots, -1)$ 
have a density with respect to Lebesgue measure on $W_n^+$ given by
$\bar{r}_t(x, y)$.

\begin{lemma}
\label{schutz_brownian_lemma}
For any $f : W_n^+ \rightarrow \mathbb{R}$ which is bounded, 
continuous and zero in a neighbourhood 
of the boundary of $W_n^+$,
\begin{equation*}
 \lim_{t \rightarrow 0} \int_{W_n^+} r_t(\mathbf{x}, \mathbf{y}) 
 f(\mathbf{y}) d\mathbf{y} = f(\mathbf{x})
\end{equation*}
uniformly for all $\mathbf{x} \in W_n^+$. This also holds with $r$ 
replaced by $\bar{r}$.
\end{lemma}

Let
$\mathcal{G}_{x_k} = \frac{1}{2} \frac{d}{dx^2} - \alpha_k \frac{d}{dx}$ 
denote the generator of a Brownian motion with drift $-\alpha_k$
and $\mathcal{G} = \sum_{k=1}^n \mathcal{G}_{x_k}$.

\begin{proof}[Proof of Proposition \ref{transition_density_Brownian}]
We show that $r$ satisfies the Kolmogorov backward
equations, together with its boundary conditions,  for the 
process $Y = (Y_1, \ldots, Y_n)$. 
 Let 
 \begin{equation*}
q(t; \mathbf{x, y}) =   \text{det}(D_{y_j}^{\alpha_1 \ldots \alpha_j} 
J_{x_i}^{-\alpha_1 \ldots -\alpha_i}
 \psi_t (x_i, y_j))_{i, j = 1}^n
 \end{equation*}
 and observe that 
 \begin{equation*}
 \frac{\partial r}{\partial x_i} = e^{-\sum_{i=1}^n \alpha_i(y_i - x_i) - \alpha_i^2 t/2}
 D_{x_i}^{-\alpha_i} q = 0 \text{ at } x_i = x_{i-1}
 \end{equation*}
 because the $i$-th and $(i-1)$-th rows of the determinant defining  $D_{x_i}^{-\alpha_i} q$ coincide
 at $x_i = x_{i-1}$, by virtue of  the identity $ D_{x_i}^{-\alpha_i} J_{x_i}^{-\alpha_i} f = -f.$
 
To show that $\partial r/\partial x_1 = 0$ at $x_1 = 0$ 
we consider the matrix in the definition of $r$ 
and bring the prefactor $e^{\alpha_1 x_1}$ in $r$
into the top row of this matrix.
We use the identity
$e^{\alpha_1 x_1} J^{-\alpha_1}_{x_1} \psi_t(x_1, y_j) = 
J^0_{x_1} e^{\alpha_1 x_1} \psi_t(x_1, y_j)$ and observe that the derivative in $x_1$
of the right hand side equals zero when evaluated at $x_1 = 0$.  This shows that the derivative of every 
term in the top row of this matrix equals zero because
the derivative in $x_1$ commutes with 
the operations acting in $y_j$. Therefore $\partial r/\partial x_1 = 0$ at $x_1 = 0$.
 

To show that the Kolmogorov backward equation is satisfied for 
$x, y$ in the interior of $W_n^+$ 
we let $r_{ij}(t; x_i, y_j)= e^{\alpha_i x_i - \alpha_i^2 t/2} D_{y_j}^{\alpha_1 \ldots \alpha_j}
J_{x_i}^{-\alpha_1, \ldots, -\alpha_i} \psi_t(x_i, y_j)$.
We
differentiate in $t$, and use the fact that $\psi_t$ satisfies the heat equation,  to obtain  
\begin{equation*}
\frac{\partial r_{ij}(t; x_i, y_j)}{\partial t} = 
e^{\alpha_i x_i - \alpha_i^2 t/2}  D_{y_j}^{\alpha_1 \ldots \alpha_j}
J_{x_i}^{-\alpha_1, \ldots, -\alpha_i}
  \left( \frac{1}{2} \frac{\partial^2\psi_t(x_i, y_j)}{\partial x_i^2} - 
   \frac{1}{2} \alpha^2 \psi_t(x_i, y_j)\right).
\end{equation*}
It is convenient to express the terms in brackets
using the operations $D$ and $J$,
\begin{equation*}
\left( \frac{1}{2} \frac{\partial^2\psi_t(x_i, y_j)}{\partial x_i^2} - 
   \frac{1}{2} \alpha^2 \psi_t(x_i, y_j)\right) =  \frac{1}{2}
   D_{x_i}^{\alpha_i} D_{x_i}^{-\alpha_i} \psi_t(x_i, y_j). 
\end{equation*}
The operations $J_x$ and $D_x$ commute and therefore
 \begin{IEEEeqnarray*}{rCl}
\frac{\partial r_{ij}(t; x_i, y_j)}{\partial t} & = & \frac{1}{2}e^{\alpha_i x_i - \alpha_i^2 t/2} 
 D_{x_i}^{\alpha_i} D_{x_i}^{-\alpha_i}
  D_{y_j}^{\alpha_1 \ldots \alpha_j}  J_{x_i}^{-\alpha_1, \ldots, -\alpha_i}
  \psi_t(x_i, y_j) 
  \\ 
  & = & \frac{1}{2}
  e^{\alpha_i x_i}  D_{x_i}^{\alpha_i} D_{x_i}^{-\alpha_i} e^{-\alpha x_i} 
  r_{ij}(t; x_i, y_j) \\
  & = & \mathcal{G}_{x_i} r_{ij}(t; x_i, y_j).
\end{IEEEeqnarray*}
Therefore, since $r_t(x,y)= e^{-\sum \alpha_i y_i}\text{det}(r_{ij}(t; x_i, y_j))$,
\begin{equation*}
 \frac{\partial r}{\partial t} = \sum_{i = 1}^n \mathcal{G}_{x_i} r.
\end{equation*}
The proposed transition densities $r$ satisfy the Kolmogorov backward equation 
for $Y = (Y_1, \ldots, Y_n)$ and the arguments in \cite{warren}
show that $r$ are the transition densities for $Y$. We sketch this argument but refer 
to \cite{warren} for the details.
Let $f$ be a bounded continuous function which is zero in a neighbourhood of 
the boundary of $W_n^+$ and define $F(u, x) = \int_{W_n^+} r_t(x, y) f(y) dy$.
Fix some $T > 0$ and $\epsilon > 0$. 
By using It\^{o}'s formula and the fact that $r_t$ solves the 
Kolmogorov backward equation we obtain that
$(F(T+\epsilon -t, Y_t) : t \in [0, T])$ is a martingale with respect to the 
process $(Y_t)_{t \geq 0}$.
In particular,
$F(T+\epsilon, y) = E_y(F(\epsilon, Y_T))$. 
The $\epsilon$ is introduced in order to ensures 
smoothness of $F$ and allow the 
application of It\^{o}'s formula, however, using Lemma \ref{schutz_brownian_lemma} 
we can take the limit as $\epsilon$ 
tends to zero
to conclude that
$E_y(f(Y_T)) = \int_{W_n^+} r_T(x, y) f(y) dy$. This holds for all 
bounded continuous $f$ which are zero in a neighbourhood of 
the boundary of $W_n^+$ which is sufficient to prove
that $r_T(y, \cdot)$ is the  density of  the distribution of $Y_T$
since this distribution does not charge the boundary.
\end{proof}

\begin{proof}[Proof of Lemma \ref{schutz_brownian_lemma}]
The proof follows the argument in \cite{warren}. 
The transition density for killed Brownian motion satisfies 
$\psi_t(x, y) = \phi_t(y - x) - \phi_t(x + y)$ and we can split the determinant 
\begin{equation*}
 q(t; \mathbf{x}, \mathbf{y}) = \text{det}(D_{y_j}^{\alpha_1 \ldots \alpha_j} 
 J_{x_i}^{-\alpha_1 \ldots 
 -\alpha_i} \psi_t(x_i, y_j))_{i, j = 1}^n
\end{equation*}
into a sum of two terms $q = q_1 + q_2$ where 
\begin{equation*}
 q_1(t; \mathbf{x}, \mathbf{y}) = \text{det}(D_{y_j}^{\alpha_1 \ldots \alpha_j} 
 J_{x_i}^{-\alpha_1 \ldots 
 -\alpha_i} \phi_t(y_j - x_i))_{i, j = 1}^n
\end{equation*}
and $q_2 := q - q_1$. 
We first show that 
\begin{equation}
\label{initial_cond_conv}
 \lim_{t \rightarrow 0} \int f(\mathbf{y}) 
 e^{-\sum_i \alpha_i(y_i - x_i)} q_2(t; \mathbf{x}, \mathbf{y}) 
 d\mathbf{y} = 0.
\end{equation}
We observe that $q_2$ is a sum of products of factors where in each product 
there is at least one factor of the form 
\begin{equation}
\label{negligible_term_1}
 D_{y_j}^{\alpha_1 \ldots \alpha_j} 
 J_{x_i}^{-\alpha_1 \ldots -\alpha_i} \phi_t(x_i + y_j)
\end{equation}
for some $1 \leq i, j \leq n.$
For $\{y_1 \leq \epsilon\}$ the function $f$ takes the value zero and on 
$\{y_1 > \epsilon\}$ the factor  \eqref{negligible_term_1} is  approaching zero exponentially fast as $1/t \rightarrow \infty$. As a result \eqref{initial_cond_conv} 
holds.

We now consider $q_1$ and observe that the entries in the matrix
simplify
due to the translation invariance of the function: in particular
$D^{\alpha}_y J^{-\alpha}_x h(y - x) = 
h(y - x)$
for any smooth function $h$.
This means that the matrix in $q_1$ has diagonal entries 
\begin{equation*}
D_{y_j}^{\alpha_1 \ldots \alpha_j} J_{x_i}^{-\alpha_1 \ldots -\alpha_j}
 \phi_t (y_j-x_j) = \phi_t(y_j - x_i). 
\end{equation*}
Therefore the term corresponding to the 
identity permutation in the determinant of $q_1$ 
is a standard $n$-dimensional heat kernel. 
The remaining terms are negligible as in \cite{warren}. 
\end{proof}

The transition densities must satisfy the semigroup property and this suggests 
a generalisation of the Andr\'eief (or Cauchy-Binet) identity. This identity
states that 
for any functions $(f_i)_{i = 1}^n$ and $(g_i)_{i=1}^n$ the
convolution of two determinants is a single determinant,
\begin{eqnarray}
\label{andreief}
 \int_{W^n} \text{det}( 
 f_i(x_j))_{i, j = 1}^n\text{det}(
 g_j(x_i))_{i, j = 1}^n dx_1 \ldots dx_n 
 = \text{det}\bigg(
 \int_0^{\infty} f_i(x) g_j(x) dx\bigg)_{i, j =1}^n.
\end{eqnarray}
We prove a generalisation involving the inhomogeneous derivative and integral
operators, $J$ and $D$.
\begin{lemma}
\label{ibp_lemma_inhomogeneous}
Let $(f_i)_{i = 1}^n$ and $(g_j)_{j=1}^n$ be collections of infinitely differentiable
functions on $[0, \infty)$ such that $g_j$ has superexponential decay at infinity
for each $j = 1, \ldots, n$ 
while $f_i$ has at most exponential growth at infinity 
for each $i = 1, \ldots, n$.
\begin{enumerate}[(i)]
 \item For $k \geq 1$, let 
$g^{(-k)}(x) =  \int_x^{\infty} \frac{(x - u)^{k-1}}{(k-1)!}
 g(u) du$ and $f^{(k)}$ denote 
 the $k$-th derivatives of $f$. Then\begin{equation*}
\int_{W_n^+} \text{det}(f_{i}^{(j-1)}(x_j))_{i, j =1}^n 
\text{det}(g_{j}^{(-i+1)}(x_i))_{i, j = 1}^n dx_1 \ldots dx_n 
= \text{det}\left(\int_0^{\infty} f_{i}(x) g_{j}(x) dx \right)_{i, j = 1}^n
 \end{equation*}
 \item Let $D^{\alpha}, J^{\alpha}$ be defined as in equation (\ref{J_D})
 and assume $f_i(0) = 0$ for each $i = 1, \ldots, n$.
 Then\begin{eqnarray*}
 \int_{W_n^+} \text{det}\big(D^{\alpha_1, \ldots, \alpha_j} 
 f_i(x_j)\big)_{i, j = 1}^n\text{det}\big(
 J^{-\alpha_1, \ldots, -\alpha_i}
 g_j(x_i)\big)_{i, j = 1}^n dx_1 \ldots dx_n 
 = \text{det}\bigg(
 \int_0^{\infty} f_i(x) g_j(x) dx\bigg)_{i, j = 1}^n.
\end{eqnarray*}
\end{enumerate}
%
\end{lemma}
We note that (i) is not quite the homogeneous case of (ii) because (ii) involves 
applying integration by parts to $x_1$, whereas (i) does not. We have not intended 
to make the conditions on $g$ optimal and have simply chosen some conditions 
which are sufficient for our purposes.
\begin{proof}
We start with the proof of (ii).
We observe that for $0 \leq x < z$,
\begin{equation}
\label{ibp}
 f(z) J^{-\alpha} g(z) - f(x) J^{-\alpha} g(x) = \int_x^z D^{\alpha} f(y) J^{-\alpha}
 g(y) dy - \int_x^z f(y) g(y) dy.
\end{equation}

%
%
We use this formula iteratively to prove that 
\begin{eqnarray}
\label{ibp_eq}
 \int_{W^+_n} \text{det}\big(D^{\alpha_1, \ldots, \alpha_j} 
 f_i(x_j)\big)\text{det}\big(
 J^{-\alpha_1, \ldots, -\alpha_i}
 g_j(x_i)\big) dx_1 \ldots dx_n 
 = \int_{W^+_n} \text{det}\big(
 f_i(x_j)\big)\text{det}\big(
 g_j(x_i)\big) dx_1 \ldots dx_n. 
\end{eqnarray}
For the first step we use a Laplace expansion of the determinants appearing on the left hand side and then apply
equation (\ref{ibp}) with  parameter
$\alpha=\alpha_n$ and integrating with respect to $x_n$ from $x_{n-1}$ to $\infty$. Then we reconstruct the resulting expressions as determinants. 
This gives three terms. The first term is
\begin{equation*}
 \int_{W^+_n} \text{det}(F_{ij}(x_j))_{i, j= 1}^n \text{det}(G_{ij}(x_i))_{i, j =1}^n
 dx_1 \ldots dx_n
\end{equation*}
where $ F_{ij}(x_j) = D^{\alpha_1 \ldots \alpha_j} f_{i}(x_j)$
for all $1 \leq i \leq n$ and $1 \leq j \leq n-1$, 
$F_{in}(x_n) = D^{\alpha_1 \ldots \alpha_{n-1}} f_{i}(x_n)$ for all 
$1 \leq i \leq n$, $G_{ij}(x_i) = J^{-\alpha_1 \ldots - \alpha_i} g_j(x_i)$
for all $1 \leq i \leq n-1$ and $1 \leq j \leq n$,
and $G_{nj}(x_n) = J^{-\alpha_1 \ldots - \alpha_{n-1}} g_j(x_{n})$
for all $1 \leq j \leq n$.
The other two terms are boundary terms given by
the following expression evaluated at $x_n = x_{n-1}$ 
and $x_{n} = \infty$,
\begin{equation*}
 \int_{W^+_{n-1}} \text{det}(A_{ij}(x_j))_{i, j= 1}^n \text{det}(B_{ij}(x_i))_{i, j =1}^n
 dx_1 \ldots dx_{n-1}
\end{equation*}
where  $A_{ij}(x_j) = D^{\alpha_1 \ldots \alpha_j} f_{i}(x_j)$
for all $1 \leq i \leq n, 1 \leq j \leq n-1$,
$A_{in} = D^{\alpha_1 \ldots \alpha_{n-1}} f_{i}(x_n)$ for all $1 \leq i \leq n$
and $B_j(x_i) = J^{-\alpha_1 \ldots - \alpha_i} g_j(x_i)$ for all $1 \leq i, j 
               \leq n$.   
These boundary terms are both zero: the determinant of $A_{ij}$ vanishes at 
$x_{n} = x_{n-1}$, because two columns are equal,  and we obtain zero 
at infinity by virtue of  the growth and decay conditions imposed on $f$ and 
$g$.
               
The general structure becomes clear after the second step. 
We perform the same procedure with the integration by parts (\ref{ibp}) 
with parameter $\alpha=\alpha_{n-1}$, and integrating with respect to the  variable $x_{n-1}$ between $x_{n-2}$ and $x_n$.  
We obtain three terms as above with  
\begin{align*}
  F_{ij}(x_j) & =  \begin{cases}
               D^{\alpha_1 \ldots \alpha_j} f_{i}(x_j) & \quad \qquad \text{ for all } 1 \leq i 
               \leq n,  \qquad 1 \leq j \leq n-2 \\
               D^{\alpha_1 \ldots \alpha_{j-1}} f_{i}(x_j) 
              & \quad \qquad \text{ for all } 1 \leq i 
               \leq n, \qquad n-1 \leq j \leq n 
               \end{cases}\\
 G_{ij}(x_i)  & = \begin{cases}
            J^{-\alpha_1 \ldots - \alpha_i} g_j(x_i) & \qquad \text{ for all } 1 \leq i 
               \leq n-2, \qquad 1 \leq j \leq n \\
                 J^{-\alpha_1 \ldots - \alpha_{i-1}} g_j(x_{i}) & \qquad \text{ for all } 
                n-1 \leq i \leq n,  \qquad 1 \leq j \leq n-1.
           \end{cases}
\end{align*}

and the boundary terms evaluated at $x_{n-1} = x_{n-2}$ and $x_{n-1} = x_n$ with 
               \begin{align*}
 A_{ij}(x_j) &= \begin{cases}
               D^{\alpha_1 \ldots \alpha_j} f_{i}(x_j) & \qquad \text{ for all } 1 \leq i 
               \leq n, \qquad 1 \leq j \leq n-2 \\
               D^{\alpha_1 \ldots \alpha_{j-1}} f_{i}(x_j) & \qquad \text{ for all } 1 \leq i 
               \leq n, \qquad n-1 \leq j \leq n
              \end{cases}\\
B_{ij}(x_i) & =  \begin{cases}
          J^{-\alpha_1 \ldots - \alpha_i} g_j(x_i) & \quad \text{ for all } 1 \leq i 
               \leq n-1,\qquad 1 \leq j \leq n  \\    
               J^{-\alpha_1 \ldots - \alpha_{i-1}} g_j(x_i) & \quad \text{ for } i = n, 
              \qquad 1 \leq j \leq n.
                 \end{cases}
\end{align*}
The determinant of $A_{ij}$ will vanish at $x_{n-1} = x_{n-2}$ while the determinant
of $B_{ij}$ will vanish at $x_{n-1} = x_n$. Therefore both boundary terms vanish. 
Equation (\ref{ibp_eq}) now follows by iterating this procedure. The order of the 
integration by parts with respect to the variables 
and choice of the parameter  $\alpha$ in (\ref{ibp}) is important to ensure 
there are no boundary terms and is the following:
$(x_n, \alpha_n), (x_{n-1}, \alpha_{n-1}), \ldots,
(x_1, \alpha_1)$ then $(x_n, \alpha_{n-1}), (x_{n-1}, \alpha_{n-2}), \ldots,
(x_2, \alpha_1)$ until finally $(x_n, \alpha_1)$. 
In the integration by parts with respect to $(x_1, \alpha_1)$ there is  
a boundary term at zero, however, this is also zero 
due to the constraint that $f_i(0) = 0$ for each 
$i = 1, \ldots, n$.

Finally part (ii) of the lemma follows from applying the Andr\'eief identity \eqref{andreief} to the righthand side of  equation $(\ref{ibp_eq})$. 
Part (i) of the Lemma is the same except that there is no integration by 
parts in $x_1$ so that the condition $f_i(0) = 0$ is not required.
\end{proof}

\subsection{Invariant Measures}
\begin{lemma}[Dupuis and Williams \cite{dupuis_williams}] 
\label{exist_unique}
 Let $(Y_1(t), \ldots, Y_n(t))_{t \geq 0}$ be the system of reflected Brownian 
 motions with a wall given in equation (\ref{sdes}). 
  Then $Y$ has a unique invariant measure denoted $\pi$ and satisfies
$\norm{P_t(x, \cdot) - \pi} \rightarrow 0$ for all $x \in W_n^+$
where $\norm{\mu} = \sup_{\lvert g \rvert \leq 1} \lvert \int \mu(dy)
 g(y) \rvert$ is the total variation distance of $\mu$.
\end{lemma}
There are stronger results in the literature including convergence rates: 
for example Theorem 4.12 of \cite{budhijara_lee}
can be applied to prove $V$-uniform ergodicity for $Y$.
For the process where all particles are
started from the origin, 
the convergence in distribution is contained in Proposition \ref{time_change_refl}.

\begin{proposition}
\label{invariant_measure}
\begin{enumerate}[(i)]
 \item When $\alpha_1 = \ldots = \alpha_n = 1$, then $(Y_1^*, \ldots, Y_n^*)$ has 
 a density with respect to Lebesgue measure on $W_n^+$ given by
\begin{equation}
\label{eq_psi}
\bar{\pi}(x_1, \ldots, x_n) = \text{det}(f_{i-1}^{(j-1)}(x_j))_{i, j = 1}^n
\end{equation}  
with the  sequence of functions $(f_i)_{i \geq 0}$ 
defined inductively as follows: 
\begin{align}
\label{f_defn}
f_0(x) & = 2e^{-2x}\\
\mathcal{G}^* f_{i+1} & = f_i \text{ and } f_i'(0) = f_i(0) = 0
\text{ for } i \geq 1
\end{align}
where $\mathcal{G}^* = \frac{1}{2}\frac{d^2}{dx^2} + \frac{d}{dx}$.
\item  
When the drifts are distinct, $(Y_1^*, \ldots, Y_n^*)$ has 
 a density with respect to Lebesgue measure on $W_n^+$ given by
 \begin{equation*}
  \pi(x_1, \ldots, x_n) = \frac{1}{\prod_{i < j} (\alpha_i - \alpha_j)}e^{- \sum_{i=1}^n \alpha_i x_i} \text{det}(D^{\alpha_1 \ldots \alpha_j}
  f_i(x_j))_{i, j = 1}^n
 \end{equation*}
 where $f_i(x) = e^{\alpha_i x} - e^{-\alpha_i x}$.
\end{enumerate}
\end{proposition}

We make two remarks:
\begin{enumerate}[(i)]
 \item For equal drifts the initial function $f_0$ satisfies
$\mathcal{G}^* f_0 = 0$ and $f_0'(0) + 2f_0(0) = 0$. 
The functions $f_i$ could also have been defined so as to satisfy the 
boundary condition $f_i'(0) + 2f_i(0) = 0$ for $i \geq 1$, however, 
$\psi$ would be unchanged as we can use row operations to add on
constant multiples of $f_0$.
\item When the drifts are distinct, Dieker and Moriarty \cite{dieker2009}
show the invariant measure is a sum of 
exponential random variables and this sum can be calculated explicitly 
for small values of $n$. However, when the drifts 
coincide Proposition \ref{invariant_measure} part (i) shows the invariant measure contains 
polynomial prefactors in the style of repeated eigenvalues.
\end{enumerate}

\begin{lemma}
\label{positivity}
 The functions $\bar{\pi}$ and $\pi$ are positive on $W_n^+$ and satisfy $\int_{W_n^+} \bar{\pi} = \int_{W_n^+} \pi = 1$. 
\end{lemma}
We will prove this in Section \ref{lpp} and for the moment prove 
Theorem \ref{equality_law} assuming this Lemma.

\begin{proof}[Proof of Proposition \ref{invariant_measure}]
In the case of equal rates we apply part (i) of Lemma \ref{ibp_lemma_inhomogeneous} 
to calculate 
the convolution between the proposed invariant measure and the transition densities 
from Proposition \ref{transition_density_Brownian}. 
The functions $f_i$ and $\eta$ satisfy the growth and decay 
conditions at infinity for Lemma \ref{ibp_lemma_inhomogeneous} 
and this shows that
\begin{equation*}
 \int_{W_n^+} \bar{\pi}(\mathbf{x}) \bar{r}_t(\mathbf{x}, \mathbf{y})  d\mathbf{x}
 = \text{det}\left(D^{(j-1)}_{y_j}\int_0^{\infty} f_{i-1}(x) 
 \eta_t^{(-1)}(x, y_j) dx\right)_{i, j = 1}^n
\end{equation*}
where $D^{(j-1)}$ denotes the $j$-th iterated concatenation of $D^0$ and 
$ \eta_t^{(-1)}$ is the transition density of reflected Brownian motion 
with drift $-1$.
Fixing $y$, we use the notation \begin{equation*}
     (f_i, \eta_t^{(- 1)}) = \int_0^{\infty} f_i(x) 
     \eta_t^{(-1)}(x, y) dx.
    \end{equation*}
Let $\mathcal{G} = \frac{1}{2}\frac{d^2}{dx^2} - \frac{d}{dx}$ and then for $k\geq 1$, since $ \frac{d}{dt}\eta_t^{(-1)}=\mathcal{G}\eta_t^{(-1)}$,
\begin{equation*}
 \frac{d}{dt}(f_k, \eta_t^{(-1)})
 = (f_k, \mathcal{G} \eta_t^{(-1)}) = (\mathcal{G}^*
 f_k, \eta_t^{(-1)})
 = (f_{k-1}, \eta_t^{(-1)}).
\end{equation*}
The step $ (f_k, \mathcal{G} \eta_t^{(-1)}) 
= (\mathcal{G}^* f_k, \eta_t^{(-1)})$ 
follows from integrating by parts where
the boundary terms are given by $f_k(x) \frac{d}{dx} 
\eta_t^{(-1)}(x, y)$ and 
$\eta_t^{(-1)}(x, y) (\frac{df_k}{dx} + 2f_k(x))$
each evaluated at zero and infinity. The boundary terms all equal to 
zero by the boundary conditions on $\eta$ and $f_k$.
Integrating in $t$, 
\begin{equation*}
 (f_k, \eta_t^{(-1)}) = f_k(y)+ \int_0^t   (f_{k-1}, \eta_t^{(-1)}) ds\end{equation*}
and iterating this gives, since $ (f_0, \eta_t^{(-1)})=f_0(y)$,
\begin{equation*}
 (f_k, \eta_t^{(-1)}) = \frac{t^k}{k!} f_0(y) + \ldots + f_{k}(y).
\end{equation*}
Thus the functions $f_k$ are  invariant  under the action of the $\eta_t^{(-1)}$ 
modulo multiples of $f_0, \ldots, f_{k-1}$.
Consequently, for any $t > 0$ we can  apply row operations to obtain
\begin{equation*}
 \int_{W_n^+} \bar{\pi}(\mathbf{x}) \bar{r}_t(\mathbf{x}, \mathbf{y})  d\mathbf{x} = 
 \text{det}(f_{i-1}^{(j-1)}(y_j))_{i,j=1}^n
 = \bar{\pi}(\mathbf{y}).
\end{equation*}

In the case when the drifts are not equal we apply
Lemma \ref{ibp_lemma_lpp} to express the convolution 
of our proposed invariant measure and the transition density from Proposition 
\ref{transition_density_Brownian} as 
a single determinant,
\begin{equation*}
 \int \pi(\mathbf{x}) r_t(\mathbf{x, y}) d\mathbf{x}
 = e^{-\sum_{i=1}^n \alpha_i y_i} \text{det}\left(D_y^{\alpha_1 \ldots \alpha_j}
 \int_0^{\infty} f_i(x) \psi_t(x, y_j) e^{-\alpha_i^2 t/2} dx
 \right)_{i, j = 1}^n.
\end{equation*}
The conditions for Lemma \ref{transition_density_Brownian} are satisfied because
$f_i(0) = 0$ for each $i = 1, \ldots, n$ and the conditions on the growth and decay of $f_i$ and $\psi$ at 
infinity are satisfied. 
We have 
\begin{equation*}
  \int_0^{\infty} f_i(x) \psi_t(x, y) e^{-\alpha_i^2 t/2} dx
  = f_i(y)
\end{equation*}
and therefore 
\begin{equation*}
 \int \pi(\mathbf{x}) r_t(\mathbf{x, y}) d\mathbf{x}
 = \frac{1}{\prod_{i < j} (\alpha_i - \alpha_j)}e^{-\sum_{i=1}^n \alpha_i y_i} \text{det}\left(D_y^{\alpha_1 \ldots \alpha_j}
f_i(y_j)
 \right)_{i, j = 1}^n = \pi(\mathbf{y}).
\end{equation*}
\end{proof}

\section{Point to line last passage percolation}
\label{lpp}

\subsection{Transition densities}
Last passage percolation times can be interpreted as 
an interacting particle system with a pushing interaction between the 
particles. We define a Markov chain $(\mathbf{G}^{\text{pp}}(k))_{k \geq 0}$
with $n$ particles with positions on the real line ordered as
$G_1^{\text{pp}} < \ldots < G_n^{\text{pp}}$. We update the system between time $k-1$
and time $k$ by applying the following local update rule sequentially 
to $G_1^{\text{pp}}, \ldots, G_n^{\text{pp}}$ as follows:
\begin{equation}
\label{update_rule_lpp}
 G_j^{\text{pp}}(k) = \max\{G_j^{\text{pp}}(k-1), G_{j-1}^{\text{pp}}(k))\}
 + e_{jk}
\end{equation}
where $(e_{jk})_{1 \leq j \leq n, k \geq 1}$ are an independent sequence of exponential 
random variables and $G_1^{\text{pp}}(0) = \ldots = G_n^{\text{pp}}(0)  = 0$. The
interactions of the particles are exactly the local 
update rules of last passage percolation and the largest particle position at time $n$
describes the point-to-point last passage percolation time 
$G^{\text{pp}}_{n} (n) = \max_{\pi \in \Pi_n} \sum_{(ij) \in \pi} e_{ij}$
where $\Pi_n$ is the set of all directed (up and right) 
paths from $(1, 1)$ to $(n, n)$.

The advantage of such an interpretation is that there is an explicit
transition 
density for this Markov chain. This was proven
in the case of equal 
parameters (and geometric data) 
by Johansson \cite{johansson2010} 
and with inhomogeneous parameters (and geometric data) 
by Dieker and Warren \cite{dieker2008}. 
This Markov chain plays an important role in the recent work, for example \cite{johansson_rahman}, on the two-time distribution 
of last passage percolation. 
In this section we show how this Markov chain can also be used to study 
point-to-line last passage percolation.

For $\alpha \in \mathbb{R}$, let $D^{\alpha}, I^{\alpha}$ be defined 
by acting on functions $f : \mathbb{R} \rightarrow \mathbb{R}$ 
which are infinitely differentiable for $x > 0$, are equal to zero on $x \leq 0$ 
and satisfy that $f^{(k)}(0_+)$ exists for each $k \geq 0$.
On such a class of functions define 
\begin{equation}
\label{defnD}
 D^{\alpha} f(x)  = \begin{cases} 
                    f'(x) - \alpha f(x), & x > 0 \\
                    0, & x \leq 0 
                   \end{cases}
                   \qquad\qquad
 I^{\alpha} f(x)  = \begin{cases}
                      \int_0^x e^{\alpha(x - t)} f(t) dt, & x > 0 \\
                      0, & x \leq 0.
                     \end{cases}
                 \end{equation}
Then $D^{\alpha}, I^{\alpha}$ preserve this class of functions and satisfy 
$D^{\alpha} I^{\alpha} f = f$ for functions of this form. 
We also define homogeneous analogues: for a function $g$ satisfying the above,
define $g^{(r)}(x)$ or $D^{(r)} g$ to be the $r$-th iterated derivative of $g$ for $x > 0$ and 
equal to zero for $x \leq 0$ and similarly $g^{(-r)}(x)$ or $I^{(-r)} g$ 
to be
the iterated integral $\int_0^x \frac{(x - y)^{r - 1}}{(r - 1)!} g_m(y) dy$ for 
$x > 0$ and equal to zero for $x \leq 0$.

\begin{proposition}
\label{transition_densities_lpp}
Let $(\mathbf{G}^{pp}(k))_{k \geq 0}$ be the Markov chain described above with $n$ particles 
constructed 
from independent exponentially distributed random variables $(e_{ij})_{1 \leq i \leq n, j \geq 1}$ 
with $e_{ij}$ having rate $\alpha_i > 0$. 
\begin{enumerate}[(i)]
 \item In the case of equal rates: $\alpha_1 = \ldots = \alpha_n = 2$,  
 the $m$-step transition probabilities
have a density with respect to Lebesgue measure 
on $W_n^+$ given by, for $x, y \in W_n^+$,
\begin{equation*}
Q_m(x, y) = \text{det}(g_m^{(j-i)}(y_j - x_i))_{i, j = 1}^n
\end{equation*}
where $g_m(z)  = \frac{2^m}{\Gamma(m)} z^{m-1} e^{-2z} 1_{z > 0}$ and $g_m^{(r)}$
are 
defined above. 
\item For $\alpha_j > 0$ for each $j = 1, \ldots, n$, the $m$-step transition densities 
 have a density with respect to Lebesgue measure on 
 $W_n^+$ given by, for $x, y \in W_n^+$, 
 \begin{equation*}
  Q_m(x, y) = \left( \prod_{i=1}^n \alpha_i \right) e^{-\sum_{i=1}^n\alpha_i(y_i - x_i)}
  \text{det}(f_m^{(i, j)}(y_j - x_i))_{i, j = 1}^n
 \end{equation*}
where $f_m(u) = \frac{u^{m-1}}{(m-1)!} 1_{u > 0}$ and
\begin{equation}
\label{f_deriv_int}
 f_m^{(i, j)}(z) = \begin{cases}
     D^{\alpha_{i+1} \ldots \alpha_{j}}   f_m(z) & \text{ for } j > i \\
   I^{\alpha_{j+1} \ldots \alpha_i}   f_m(z) & \text{ for } j <  i \\
               f_m(z) & \text{ for } i = j.
                                                         \end{cases}
\end{equation}
with $D$ and $I$ defined 
in equation (\ref{defnD}).
\end{enumerate}
\end{proposition}
Our proof is a generalisation of the method in Johansson \cite{johansson2010} to the case of 
inhomogeneous parameters and exponential rather than geometric jump distributions.
An exponential case is not an entirely straightforward generalisation of the formulas in the 
geometric case because of taking derivatives of functions with a discontinuity.
In order to obtain $m$-step transition densities from $1$-step transition densities 
we prove a version of Lemma \ref{ibp_lemma_inhomogeneous} for our operators $D$ and $I$. There are two 
differences: we allow for possible discontinuities in the functions at the origin and
part (ii) of the Lemma allows for new particles to be added at the origin. 
This will be used in the next subsection to study point-to-line last passage 
percolation.
\begin{lemma}
\label{ibp_lemma_lpp}
\begin{enumerate}[(i)]
\item Let $f, g$ be functions satisfying the 
conditions at the start of this section. Then 
for $x, z \in W_n^+$,
\begin{equation*}
 \int_{W_n^+} \text{det}\big(
 f^{(i, j)}(y_j-x_i)\big)_{i, j = 1}^n\text{det}\big(
 g^{(i, j)}(z_j-y_i)\big)_{i, j = 1}^n dy_1 \ldots dy_n 
 = \text{det}\bigg(
(f*g)^{(i, j)}(z_j - x_i)
\bigg)_{i, j = 1}^n
\end{equation*}
where $(f*g)(z) = \int_{0}^{z} f(y) g(z-y) dy$ and $f^{(i, j)}$, $g^{(i, j)}$ 
and $(f*g)^{(i, j)}$ are defined analogously to \eqref{f_deriv_int}.
\item Let $(f_i)_{i=1}^{n-1}$ be a collection of infinitely differentiable functions 
on $\mathbb{R}_+$ with $f_i(0) = 0$ for each $i = 1, \ldots, n-1$. 
Let $g$ be a 
function satisfying the 
conditions at the start of this section. 
 Then for $z \in W_n^+$, and using the notation $y_1:=0
 $\begin{multline*}
 \int_{W_{n-1}^+} \text{det}\big(f_{i-1}^{(1, j)}(y_j)\big)_{i, j = 2}^n\text{det}\big(
 g^{(i, j)}(z_j-y_i)\big)_{i, j = 1}^n dy_2 \ldots dy_n \\
 = \text{det}\left(
  \begin{matrix}  
 g(z_1) &  g^{(1, 2)}(z_2) & \ldots & g^{(1, n)}(z_n) \\
 (f_1*g)(z_1) & (f_1*g)^{(1, 2)}(z_2) & \ldots &  (f_1*g)^{(1, n)}(z_n) \\
 \vdots & \vdots & \ddots & \vdots \\
 (f_{n-1}*g)(z_1) & (f_{n-1}*g)^{(1, 2)}(z_2) & \ldots & (f_{n-1}*g)^{(1, n)}(z_n) 
 \end{matrix}\right)_{i,j =1}^n
\end{multline*} 
where $(f*g)(z) = \int_0^{z} f(z-y) g(y) dy$ and $f^{(i, j)}$, 
$g^{(i, j)}$ and  $(f*g)^{(i, j)}$ all defined analogously to \eqref{f_deriv_int}.
\end{enumerate}
\end{lemma}

\begin{proof}[Proof of Proposition \ref{transition_densities_lpp}]
 We first prove that the one-step transition densities are given by $r_1$.  This
 is equivalent to showing that for all $n \geq 1$, and for $x, y \in W_n^+$, 
 \begin{equation}
 \label{lpp_transition_identity}
  e^{-\sum_{i=1}^n \alpha_i(y_i - x_i)} 
  \text{det}(f_1^{(i, j)}(y_j - x_i))_{i, j = 1}^n 
  = \prod_{j=1}^n e^{-\alpha_j(y_j - \max(x_j, y_{j-1}))} 1_{y_j > x_j}
 \end{equation}
where we use the convention $y_0 :=0$. 
The right hand side is zero unless $x_j < y_j$ for all $j = 1, \ldots, n$. We check this for the left hand side. If $y_k \leq x_k$ for 
some $1 \leq k \leq n$ then the first $k$ columns of the matrix in \eqref{lpp_transition_identity} only have non-zero elements in the first 
$k-1$ rows since for $j \leq k$ and $i \geq k$ the $(i, j)$-th entry of the matrix in \eqref{lpp_transition_identity} 
is a function which only takes non-zero values for positive arguments and the argument is $y_j - x_i \leq 0$.  

For the remainder of the proof, we can 
suppose $x_j < y_j$ for $j=1, \ldots, n$. We prove \eqref{lpp_transition_identity} by induction on $n$ and observe that the result holds at $n =1$. 
For the inductive step we use a Laplace expansion of the 
 determinant in the last row
 \begin{equation}
 \label{laplace_exp_det}
  \text{det}(f^{(i, j)}_1(y_j - x_i))_{i, j = 1}^n \\
  = \sum_{k = 1}^n (-1)^{k+n} f^{(n, k)}_1(y_k - x_n)
   \text{det}(f^{(i, j)}_1(y_j - x_i))_{i \neq n, j \neq k}.
 \end{equation}
We prove the terms in the sum for $1 \leq k \leq n - 2$ are zero by considering separately
the cases $y_k \leq x_n$ 
and $y_k > x_n$. 
If $y_k \leq x_n$ then $f^{(n, k)}_1(y_k - x_n) = 0$. Suppose instead $y_k > x_n$.
Observe that for $z > 0$ and $ j > 1$, 
\begin{equation}
\label{f_identity}
 \left(\frac{d}{dz} - \alpha_{j}\right) f^{(i,j-1)}(z) = f^{(i,j)}(z).  
\end{equation}
Since $y_k > x_n$, then \eqref{f_identity} can be used to re-express the columns indexed by $j = k+1, \ldots, n$ of the final determinant in \eqref{laplace_exp_det}
which involve strictly positive arguments $y_j - x_i$ for $j \geq k+1$. Therefore
\begin{equation}
\label{determinant_expression}
    \text{det}(f^{(i, j)}(y_j - x_i))_{i \neq n, j \neq k} \\
  =  \prod_{j=k+1}^n \left(\frac{\partial}{\partial y_j} - \alpha_j \right) \text{det}(M_{ij})_{i, j=1}^{n-1}
\end{equation}
where 
\begin{equation*}
M_{ij} = \begin{cases} 
f^{(i,j)}(y_j - x_i) & \text{ for } 1 \leq j \leq k-1 \\
f^{(i,j)}(y_{j+1} - x_i) & \text{ for } k \leq j \leq n-1.
 \end{cases}
\end{equation*}
%
We apply the inductive hypothesis 
to the determinant of $M$ with the variables $x_1, \ldots, x_{n-1}$
and $y_1, \ldots, y_{k-1}, y_{k+1}, \ldots y_{n}$ and parameters 
$\alpha_1, \ldots, \alpha_{n-1}$ to observe that (\ref{determinant_expression}) equals
\begin{equation}
\label{eq_lpp_transition}
\begin{split}
\MoveEqLeft
 \prod_{j=k+1}^n \left(\frac{\partial}{\partial y_j} - \alpha_j \right)\bigg\{
  e^{\sum_{j = 1}^{k-1} \alpha_j(y_{j} - x_j) + \sum_{j=k}^{n-1} 
  \alpha_j (y_{j+1} - x_j)} \prod_{j=1}^{k-1}
 e^{-\alpha_j(y_{j} - \max (x_j, y_{j-1}))}  \\
 & \cdot e^{-\alpha_k(y_{k+1} - \max(x_k, y_{k-1}))} 
 \prod_{j = k+1}^{n-1} e^{-\alpha_j(y_{j+1} - \max(x_j, y_j))} 
 \bigg\}.
 \end{split}
\end{equation}
We observe that $\max(y_j, x_{j}) = y_j$ 
for each $j = k+1, \ldots, n-1$. Therefore
the expression in $\{ \cdot \}$ is differentiable in $y_{k+1}, \ldots, y_n$, and 
furthermore, equals a factor of $e^{\alpha_{n-1} y_{n-1}}$ 
multiplied by a factor independent of $y_{n-1}$. Therefore the expression in $\{ \cdot \}$
vanishes once we apply $\left(\frac{\partial}{\partial y_{n-1}} - \alpha_{n-1}\right)$
and \eqref{eq_lpp_transition} equals zero.

Therefore the sum in equation (\ref{laplace_exp_det}) can be 
restricted to the sum of two terms 
\begin{equation}
\label{laplace_2}
 - f^{(n, n-1)} (y_{n-1} - x_n) 
   \text{det}(f^{(i, j)}(y_j - x_i))_{i \neq n, j \neq n-1}  \\
  +  \text{det}(f^{(i, j)}(y_j - x_i))_{i \neq n, j \neq n}.
\end{equation}
%
We consider the two cases when $y_{n - 1} \leq x_n$ and $y_{n-1} > x_n$ separately. 
If $y_{n - 1} \leq x_n$ then
the only non-zero contribution comes from the second term 
in equation (\ref{laplace_2}). 
In this case by applying the inductive hypothesis and noting that $\max(y_{n-1}, x_n) = x_n$
we obtain the required result that
\begin{equation}
\label{no_interaction_term_1}
 e^{-\sum_{j=1}^n \alpha_j(y_j - x_j)} \text{det}(f^{(i, j)}(y_j - x_i))_{i \neq n, j \neq n} = \prod_{j=1}^{n-1} e^{-\alpha_j(y_j - \max(x_j, y_{j-1}))} 
e^{-\alpha_n(y_n - x_n)}.
\end{equation}
Suppose instead $y_{n - 1} > x_n$ and consider equation \eqref{laplace_2}. 
Observe that
\begin{equation}
\label{prefactor}
 f^{(n, n-1)}(y_{n -1} - x_n) = \frac{1}{\alpha_n}(e^{-\alpha_n(y_{n-1} - x_n)} - 1).
\end{equation}
We consider the first determinant in equation 
(\ref{laplace_2}). The argument in the last column is strictly positive and so equation 
\eqref{f_identity} can be used to re-express this column as follows
\begin{equation*}
 \text{det}(f^{(i, j)}(y_j - x_i))_{i \neq n, j \neq n-1} = 
 \left( \frac{\partial }{\partial y_n} - \alpha_n \right) \text{det}(K_{ij})_{i, j = 1}^{n-1}
\end{equation*}
where 
\begin{equation*}
K_{ij} = \begin{cases} 
f^{(i,j)}(y_j - x_i) & \text{ for } 1 \leq j \leq n-2 \\
f^{(i,j)}(y_{j+1} - x_i) & \text{ for } j = n-1.
 \end{cases}
\end{equation*}
We apply the inductive hypothesis to the determinant of $K$
with variables $x_1, \ldots, x_{n-1}$ 
and $y_1, \ldots, y_{n-2}, y_n$ 
and parameters $\alpha_1, \ldots, \alpha_{n-1}$ to obtain, 
\begin{multline}
\label{det_exp}
   \text{det}(f^{(i, j)}(y_j - x_i))_{i \neq n, j \neq n-1} =
 \left( \frac{\partial}{\partial y_n}-\alpha_n\right)  
 \bigg\{e^{\sum_{j=1}^{n-2} \alpha_j(y_j - x_j) + \alpha_{n-1}
(y_n - x_{n-1})} \\  \cdot
\prod_{j = 1}^{n-2} e^{-\alpha_j(y_j - \max(x_j, y_{j-1}))} 
  e^{-\alpha_{n-1}
(y_n - \max(x_{n-1}, y_{n-2}))} 
\bigg\}. 
\end{multline}
The expression in $\{ \cdot \}$ is independent of $y_n$. Therefore the term in \eqref{det_exp}
involving $\partial/\partial y_n$ applied to $\{\cdot\}$ equals zero.

Using \eqref{prefactor}, \eqref{det_exp} and the inductive hypothesis we evaluate \eqref{laplace_2} multiplied by the 
prefactor $\exp(-\sum_{j=1}^n \alpha_j(y_j - x_j))$ for $y_{n-1} > x_n$
and obtain 
\begin{equation}
\label{no_interaction_term_2}
 \frac{1}{\alpha_n} e^{-\sum_{j=1}^n \alpha_j(y_j - x_j)} \text{det}(f^{(i, j)}(y_j - x_i))_{i \neq n, j \neq n-1} = \prod_{j=1}^{n-1} e^{-\alpha_j(y_j - \max(x_j, y_{j-1}))} 
e^{-\alpha_n(y_n - x_n)}
\end{equation}
and 
\begin{multline}
\label{interaction_term}
e^{-\sum_{j=1}^n \alpha_j(y_j - x_j)} \frac{1}{\alpha_n} e^{-(\alpha_n(y_{n-1} - x_n)} \text{det}(f^{(i, j)}(y_j - x_i))_{i \neq n, j \neq n-1}\\
= \prod_{j=1}^{n-1} e^{-\alpha_j(y_j - \max(x_j, y_{j-1}))} 
e^{-\alpha_n(y_n - y_{n-1})}.  
\end{multline}
To complete the inductive step of the proof of \eqref{lpp_transition_identity} in the case $y_{n-1} > x_n$ 
we use \eqref{laplace_exp_det} and \eqref{laplace_2} 
to simplify the left hand side of \eqref{lpp_transition_identity} and 
observe that \eqref{no_interaction_term_1} and \eqref{no_interaction_term_2} cancel 
while 
\eqref{interaction_term} equals the required expression.
This completes the inductive step and we establish that \eqref{lpp_transition_identity} holds. 

In the case when all parameters are equal, say $\alpha_1 = \ldots = \alpha_n = 2$, 
we obtain the required formula by bringing
the exponential prefactor inside the derivative 
and integral operators. 
The formula for the $m$-step transition densities follows from Lemma \ref{ibp_lemma_lpp}.
\end{proof}

\begin{proof}[Proof of Lemma \ref{ibp_lemma_lpp}]

We first prove part (i) for $f$ and $g$ which satisfy the 
 conditions of the Lemma and furthermore are infinitely 
 differentiable on all of $\mathbb{R}$.
We apply Lemma \ref{ibp_lemma_inhomogeneous} with the 
functions $f_i(\cdot) = (I^{\alpha_1, \ldots, \alpha_i} f)(\cdot - x_i)$ 
and $g_j(\cdot) = (D^{\alpha_1, \ldots, \alpha_j}g)(z_j - \cdot)$ and observe that
$D^{\alpha_1 \ldots \alpha_j} f_i(z) = f^{(i, j)}(z - x_j)$ and 
$J^{-\alpha_1, \ldots, -\alpha_i} g_j(z) = g^{(i, j)}(z - x_j)$.
The $D^{\alpha}$ have been defined 
on a more general class of functions in this section 
but agree with the definition used in Lemma \ref{ibp_lemma_inhomogeneous}
when the functions are smooth. The condition
on the growth of $f$ at infinity in Lemma \ref{ibp_lemma_inhomogeneous}
can be removed because $g(z_j - \cdot)$ is zero in a neighbourhood of infinity. 
As a result Lemma \ref{ibp_lemma_inhomogeneous} proves that 
\begin{equation}
\label{ibp_smooth}
 \int_{W_n^+} \text{det}\big(
 f^{(i, j)}(y_j-x_i)\big)_{i, j = 1}^n\text{det}\big(
 g^{(i, j)}(z_j-y_i)\big)_{i, j = 1}^n dy_1 \ldots dy_n 
 = \text{det}\left( 
(f*g)^{(i, j)}(z_j - x_i)\right)_{i, j = 1}^n
\end{equation}
where we have used the following to simplify the right hand side,
\begin{equation*}
 \int_{-\infty}^{\infty} (D^{\alpha_1 \ldots \alpha_j} 
 g)(z_j - y) (I^{\alpha_1 \ldots \alpha_i} f)(y - x_i)
 dy
 =(f*g)^{(i, j)}(z_j - x_i)
\end{equation*}
where the operators pass through the convolution
because $f$ and $g$ are smooth on all of $\mathbb{R}$.
Therefore the Lemma holds for 
functions which are infinitely differentiable on all 
of $\mathbb{R}$ in addition to satisfying the stated conditions.

%

We now use approximation to extend the class of functions $f$ and $g$ 
to those stated in the Lemma. For each $\epsilon > 0$, let $f_{\epsilon}$ be an
infinitely differentiable function satisfying 
$f_{\epsilon}(x) = f(x)$ for $x \geq \epsilon$ and $f_{\epsilon}(x) = 0$ for 
$x \leq 0$, and that there exists a constant $C$ such that $\lvert f_{\epsilon}(x) \rvert < C$ for all $\epsilon$ and all $x \in [-1, 1]$.  
For any $z \in \mathbb{R}$ and $j \geq 1$, 
\begin{gather}
\label{convergence1}
 \lim_{\epsilon \rightarrow 0} (f_{\epsilon} * g_{\epsilon})(z)  =  (f * g)(z), \qquad
  \lim_{\epsilon \rightarrow 0} I^{\alpha_1, \ldots, \alpha_j} (f_{\epsilon} * g_{\epsilon})(z)  =  I^{\alpha_1, \ldots, \alpha_j} (f * g)(z),\\
  \label{convergence2}
   \lim_{\epsilon \rightarrow 0} D^{\alpha_1, \ldots, \alpha_j} (f_{\epsilon} * g_{\epsilon})(z)  =  D^{\alpha_1, \ldots, \alpha_j} (f * g)(z). 
\end{gather}
%
%
We prove \eqref{convergence2}; equation \eqref{convergence1} is more straightforward.
Observe that if $z \leq 0$ then both sides are zero and for $z > 0$,
\begin{multline}
\label{convolution_approx}
\frac{d^{j}}{dz^{j}}\left((f_{\epsilon}*g_{\epsilon})(z) - (f*g)(z) \right) = 
\int_0^{\epsilon} f_{\epsilon}(y) g_{\epsilon}^{(j)}(z - y) dy 
- \int_0^{\epsilon} f(y) g^{(j)}(z - y) dy \\
  +  \int_0^{\epsilon} f_{\epsilon}^{(j)}(z - y) g_{\epsilon}(y)  dy
  - \int_0^{\epsilon} f^{(j)}(z - y) g(y)  dy
\end{multline}
tends to zero as $\epsilon \rightarrow 0$ because
for $\epsilon < z/2$ then $g_{\epsilon}^{(j)}(z - y) = g^{(j)}(z - y)$ 
and $f_{\epsilon}^{(j)}(z - y) = f^{(j)}(z - y)$ 
for $0 \leq y \leq \epsilon$, and $g_{\epsilon}$ and $f_{\epsilon}$ are bounded.

Equation \eqref{ibp_smooth} holds with $f$ and $g$ replaced by 
$f_{\epsilon}$ and $g_{\epsilon}$ because these are smooth. Defining 
$f_{\epsilon}^{(i, j)}$ and $g_{\epsilon}^{(i, j)}$ analogously to \eqref{f_deriv_int}
we obtain, 
\begin{multline}
\label{ibp_smooth_2}
 \int_{W_n^+} \text{det}\big(
 f_{\epsilon}^{(i, j)}(y_j-x_i)\big)_{i, j = 1}^n\text{det}\big(
 g_{\epsilon}^{(i, j)}(z_j-y_i)\big)_{i, j = 1}^n dy_1 \ldots dy_n \\
 = \text{det}\left( 
 (f_{\epsilon}*g_{\epsilon})^{(i, j)}(z_j - x_i)\right)_{i, j = 1}^n.
\end{multline}
We want to pass to the limit as $\epsilon \downarrow 0$. 
Equations \eqref{convergence1} and \eqref{convergence2} show that the right hand side of 
equation \eqref{ibp_smooth_2} converges. 

%
 
Let $x_1 < \ldots < x_n$ and $z_1 < \ldots < z_n$ and let $\epsilon < \min(\min_{i < j} \{z_j - z_i\}, \min_{i < j} \{x_j - x_i\})$. 
Consider the Laplace expansions of
the determinants on the left hand side of \eqref{ibp_smooth_2}. A term in the expansion corresponding to
permutations $\sigma$ and $\rho$ equals 
\begin{equation*}
\int_{W_n^+} \prod_{i=1}^n f_{\epsilon}^{(\sigma(i), i)}(y_i - x_{\sigma(i)})
g_{\epsilon}^{(i, \rho(i))}(z_{\rho(i)} - y_i) dy_1 \ldots dy_n.
\end{equation*}
If $\rho$ is the identity then each factor 
$g_{\epsilon}(z_{\rho(i)} - y_i)$ is bounded uniformly in $\epsilon$ for $0 \leq y_i \leq z_{\rho(i)}$. 
If $\rho$ is not the identity then there exists $i < j$ with $\rho(i) > i$ 
and $\rho(j) \leq i$. 
The $(i, \rho(i))$ factor is equal to
$g_{\epsilon}^{(i, \rho(i)}(z_{\rho(i)} - y_i)$ and is bounded uniformly in $\epsilon$ on 
the region $0 \leq y_i \leq z_{\rho(i)} - \epsilon$. On the region, 
$y_i > z_{\rho(i)} - \epsilon$ this factor may be unbounded, however, the $(j, \rho(j))$ 
factor is zero because
$y_j \geq y_i > z_{\rho(i)} - \epsilon > z_{\rho(j)}$ and therefore the argument in 
the $(j, \rho(j))$ factor is strictly negative.
The same argument applies to $\sigma$.  
This shows 
that the integrand is bounded uniformly in $\epsilon$ and since it converges pointwise 
then the convergence of the left hand side of \eqref{ibp_smooth_2}
follows from the dominated convergence theorem. 

We have established part (i) when 
$x_1 < \ldots < x_n$ and $z_1 < \ldots < z_n$. 
We will complete the proof of part (i) by showing that both sides 
are continuous in $x$ and $z$ for $x, z \in W_n^+$. 
For the right hand side of part (i), we
observe that 
$y \rightarrow (f * g)^{(i, j)}(y)$ is continuous except if $j > i$ and $y = 0$. 
We consider the Laplace expansion of the right hand side with the sum indexed by permutations $\rho$.
If $\rho$ is the identity then each factor is continuous. If $\rho$ is not the identity, then
there exists $i < j$ with $\rho(i) > i$ 
and $\rho(j) \leq i$. The argument of the $(j, \rho(j))$ factor is $z_{\rho(j)} - x_{j}$ and so the $(j, \rho(j))$ factor is zero 
on $\{z_{\rho(i)} \leq x_i\}$ because $x_j \geq x_i \geq z_{\rho(i)} \geq z_{\rho(j)}$. 
On $\{z_{\rho(i)} > x_i\}$ then the factor $(f*g)^{(i, \rho(i)}(z_{\rho(i)} - x_i)$ is continuous. 
As $z_{\rho(i)} - x_i \downarrow 0$, the factor $(f*g)^{(i, \rho(i)}(z_{\rho(i)} - x_i)$ remains bounded and 
the factor $(f*g)^{(j, \rho(j))}(z_{\rho(j)}-y_j) \rightarrow 0$.
As a result the right hand side of part (i) is continuous in $x$ and $z$. 
The integrand on the left hand side of part (i) is bounded over compact 
intervals and so the left hand side is continuous 
in $x$ and $z$.
This completes the proof of part (i).

Formally, part (ii) of the Lemma follows from embedding the matrix of size $n-1$ on the left 
hand side of part (ii) 
in a matrix of size $n$ with the addition of a delta function
\begin{equation*}
 \text{det}(f_{i-1}^{(1, j)}(y_j))_{i, j=2}^n \delta_0(y_1) = 
 \text{det}(f_{i-1}^{(1, j)}(y_j))_{i, j = 1}^n
\end{equation*}
where $f_0 := \delta_0(\cdot)$ and $f_0^{(1, j)}$ are interpreted as weak derivatives.
Continuing formally part (ii) is now an application of 
Lemma \ref{ibp_lemma_inhomogeneous}
\begin{equation*}
 \int_{W_n^+} \text{det}(D^{\alpha_2 \ldots \alpha_j} f_{i-1}(y_j))_{i, j = 1}^n
 \text{det}(J^{-\alpha_2,\ldots, -\alpha_i} g_j(y_i))_{i, j = 1}^n dy_1 \ldots dy_n
 = \text{det}\left(\int_{-\infty}^{\infty} f_{i-1}(y) g_j(y) dy\right)_{i, j = 1}^n
\end{equation*}
where $g_j(y_i) = D^{\alpha_2 \ldots \alpha_j} g(z_j - y_i)$ and 
$f_0:=\delta_0$. The top row on the right hand side is equal to 
$(\delta_0, g_j) = g^{(1, j)}(z_j)$.

To give a rigorous proof of part (ii) we use a similar integration by parts argument to Lemma 
\ref{ibp_lemma_inhomogeneous} and approximate $g$ by a smooth $g_{\epsilon}$ 
as in part (i) of the current Lemma. In the proof, the condition $f_i(0) = 0$ for each $i = 1, \ldots, n-1$
is needed for the boundary term from the integration by parts with respect 
to $y_2$ to be zero.
\end{proof}

\subsection{Proof of Theorem \ref{equality_law}}
We apply the results of the previous section to study
point-to-line last passage percolation. Recall the point to line last passage percolation times $G(k,l)$ are defined by \eqref{defnG}. 
It is convenient to view the exponential data and last passage percolation times
to be set-up in the following array:
\begin{equation*}
\begin{matrix}
 G(1, n) & \cdots & G(1, 2) & G(1, 1) \\
 & \ddots & \vdots & \vdots \\
 & & G(n-1, 2) & G(n-1, 1) \\
 & & & G(n, 1)
\end{matrix}
\end{equation*}
where we can view the vertical direction as time, increasing upwards,
 and each horizontal layer  as describing the positions  of a system of  particles with an 
additional particle added after each time step.
These last passage percolation times form a Markov chain 
$(\mathbf{G}^{\text{pl}}(k))_{1 \leq k \leq n}$ where 
$\mathbf{G}^{\text{pl}}(k) = (G(n-k+1, k), \ldots, G(n-k+1, 1))$. 
We use the notation $\mathbf{G}^{\text{pl}}(k) = (G^{\text{pl}}_1(k), \ldots, G^{\text{pl}}_k(k))$.
The recursive property of 
last passage percolation
 implies that $\mathbf{G}^{\text{pl}}$ satisfies for all $1 \leq j \leq k \leq n$,
\begin{equation}
\label{update_rule_pl}
 G_j^{\text{pl}}(k) = \max\{G_{j-1}^{\text{pl}}(k-1), G_{j-1}^{\text{pl}}(k)\}
 + e_{n-k+1, k-j+1}
\end{equation}
where we recall that $e_{ij}$ has rate $\alpha_i + \alpha_{n-j+1}$ and we use the notation $G_0^{\text{pl}}(k):=0$ for all $k = 0, \ldots, n$. 
Comparing this with the update rule for the point to point case  given at \eqref{update_rule_lpp} we see that it is the same up to a shift in  the labels of the particles.   Thus we can  repeatedly apply the $1$-step transition densities of 
Proposition \ref{transition_densities_lpp}
while adding in an extra particle at the origin after each step to  compute  the joint distribution of the vector $(G(1,n), \ldots, G(n, n))$. 
This will show that the distribution of this vector   agrees with the invariant measure of the Brownian system considered 
in Theorem \ref{invariant_measure}.
This also  proves the positivity and normalisation of 
$\bar{\pi}$ and $\pi$ stated in Lemma \ref{positivity} 
which is required to complete the proof of
Theorem \ref{invariant_measure}.

\begin{proof}[Proof of Theorem \ref{equality_law}]
We prove the result by induction on $n$  and observe that the case $n = 1$ is true. 
We first prove the case of  equal rates: $\alpha_1 = \ldots = \alpha_n = 1$.
Suppose that the distribution of $ \mathbf{G}^{\text{pl}}(n-1)$ is given by the density 
\[
\bar{\pi}(x_1, \ldots, x_{n-1})= \text{det}(f_{i-1}^{(j-1)}(x_j))_{i, j = 1}^{n-1}.
\]
where the functions $f_0, f_1,\ldots f_{n-1}$ are specified in  Proposition \ref{invariant_measure}.   In view of equation \eqref{update_rule_pl} and  Proposition \ref{transition_densities_lpp}
the distribution of 
$ \mathbf{G}^{\text{pl}}(n)$ has density given by 
\begin{equation*}  
\int_{W_{n-1}^+} \text{det}(f_{i-2}^{(j-2)}(x_j))_{i, j = 2}^n 
 \text{det}(g^{(j - i)}(y_j - x_i))_{i, j = 1}^n
 dx_2 \ldots dx_n
\end{equation*}
where we  re-label the particle positions at time $n-1$  as  $x_2, \ldots, x_n$ and use 
the notation $x_1:=0$.
We use Lemma \ref{ibp_lemma_lpp} part (ii) to express this as a single determinant
\begin{equation*}
  \text{det}\left(\begin{matrix}
 D^{(j-1)} g(y_j) & \qquad \text{ for } i = 1 \\
 D^{(j-1)} (F_{i-2}*g)  (y_j )   & \qquad \text{ for } i = 2, \ldots, n
  \end{matrix}\right)_{i, j = 1}^n
\end{equation*}
where $D^{(j)}$ denotes the $j$-th derivative, and $F_i(x)=\int_0^xf_i(z)dz$ . 
The convolutions can be calculated by using the defining property of the $f_i$, namely
that for each $i = 1, \ldots, n-1$ we have
${\mathcal G}^* f_i = f_{i-1}$ with $f_i(0) = f_i'(0) = 0$
or in integrated form for $x > 0$,
\begin{equation*}
 f_i'(x) = \int_0^x 2 e^{-2(x - u)}f_{i-1}(u) du = \int_0^x g(x - u) f_{i-1}(u) du.
\end{equation*}
From this  it follows that for $x > 0$, 
\begin{equation*}
f_i(x) =  \int_0^x F_{i-1}(u) g(x - u) du 
\end{equation*}
by differentiation  and using the boundary conditions $f_i(0) = 0$ for $i = 1, \ldots, n-1$ and $F_i(0) = 0$ for $i = 0, \ldots, n-2$. 
Finally note that $g(x)= f_0(x)$  for $x>0$.  Therefore  the distribution of 
$ \mathbf{G}^{\text{pl}}(n)$ has density given by 
\begin{equation*}
 \text{det}(f_{i - 1}^{(j - 1)}(y_j))_{i, j = 1}^n
\end{equation*}
and this completes the inductive step with equal rates.

In the case of distinct rates we proceed again by induction. 
The inductive hypothesis allows us to  suppose that the distribution of $ \mathbf{G}^{\text{pl}}(n-1)$ is given by the density 
 \begin{equation*}
 \pi(x_2, \ldots, x_n) = 
 \frac{1}{\prod_{2 \leq i < j \leq n} (\alpha_i - \alpha_j)} e^{-\sum_{i = 2}^n \alpha_i x_i} \text{det}(D^{\alpha_2 \ldots \alpha_j}
  f_i(x_j))_{i, j = 2}^n.
 \end{equation*}
Then the   distribution of  $ \mathbf{G}^{\text{pl}}(n)$  is computed using one step  transition density  for general jump  rates in 
 Proposition \ref{transition_densities_lpp} to  be 
 \begin{multline}
 \label{inductivestep}
 \frac{\prod_{j=1}^n(\alpha_1 + \alpha_j)}{\prod_{2 \leq i < j \leq n} (\alpha_i - \alpha_j)}  \int_{W_{n-1}^+}  e^{-\sum_{i = 2}^n \alpha_i x_i} 
 \text{det}(D^{\alpha_2 \ldots \alpha_j}
  f_i(x_j))_{i, j = 2}^n
  e^{-\sum_{i=1}^n (\alpha_1 + \alpha_i)(y_i - x_i)} \\
   \times \bigg( \text{det}(
   f^{(i, j; \alpha_1 + \mathbf{\alpha})}_1(y_j - x_i)
  )_{i, j = 1}^n dx_2 \ldots dx_n \bigg)
 \end{multline}
 where $f^{(i, j; \alpha_1 + \mathbf{\alpha})}_{1}$ is defined as in \eqref{f_deriv_int} but with parameters $\alpha_1 + \alpha_i$ for $i = 1, \ldots, n$
 and once again we have used the notation $x_1 :=0$.   In applying the transition density from 
 Proposition \ref{transition_densities_lpp} we need to substitute 
 $\alpha_1+\alpha_i $ for $ \alpha_i$  to take account of the fact that  the  random variable $e_{1, n-j+1}$  
 which contributes to  $  \mathbf{G}^{\text{pl}}_j(n)$ has rate $\alpha_1+\alpha_{j}$. 
 
  The exponential terms in $\alpha_1$ can be brought inside the integral in \eqref{inductivestep} and
 derivatives to obtain 
\begin{equation*}
  \int_{W_n^+} e^{-\sum_{i=1}^n \alpha_i y_i} \text{det}(D^{\alpha_2 \ldots \alpha_j}
  f_i(x_j))_{i, j = 2}^n \bigg( \text{det}(
  \hat{f}^{(i, j)}_1(y_j - x_i) 
  )_{i, j = 1}^n dx_2 \ldots dx_n \bigg)
\end{equation*}
where $\hat{f}^{(i, j)}_1$ is defined as in \eqref{f_deriv_int} but with the function $f_1$ 
replaced by $e^{-\alpha_1 z} 1_{z > 0}$
This is now in the form to apply Lemma \ref{ibp_lemma_lpp} part (ii) to obtain,
\begin{equation*}
 e^{-\sum_{i=1}^n \alpha_i y_i}  \text{det}\left(
 \begin{matrix}  
 D^{\alpha_2, \ldots, \alpha_j}  
 e^{-\alpha_1 y_j} \qquad\qquad \qquad \text{ for } i = 1 \\
 D^{\alpha_2, \ldots, \alpha_j}  
 \int_0^{y_j} f_i(x) e^{-\alpha_1(y_j - x)} dx \qquad\text{ for } i = 2, \ldots, n
 \end{matrix}\right)_{i,j =1}^n
\end{equation*}
where $D^{\emptyset} = \text{Id}$. 
%
The first row is given by
\begin{equation*}
e^{-\alpha_1 y_j} = \frac{1}{2\alpha_1}
D^{\alpha_1} f_1(x).
 \end{equation*}  
For each $i = 2, \ldots, n$ the integrals can be computed explicitly (noting that the 
$\alpha_i$ are distinct):
 \begin{IEEEeqnarray*}{rCl}
 \int_0^{y_j}  f_i(x) e^{-\alpha_1(y_j - x)} dx & = & 
 \frac{(\alpha_1 - \alpha_i)e^{\alpha_i y_j} - (\alpha_1 + \alpha_i)e^{-\alpha_i y_j}}
 {(\alpha_1 - \alpha_i)(\alpha_1 + \alpha_i)}+ Ce^{-\alpha_1 y_j} \\
&  = & \frac{1}{(\alpha_1 + \alpha_i)(\alpha_1 - \alpha_i)} D^{\alpha_1} f_i(y_j) + Ce^{-\alpha_1 y_j}
\end{IEEEeqnarray*}
where $C = C(\alpha)$ is some constant in $y_1$ and $C e^{-\alpha_1 y}$ 
can be removed from the $i$-th row by row operations.  
This shows that the density of $ \mathbf{G}^{\text{pl}}(n)$ is given by
\begin{equation*}
 \pi(x_1, \ldots, x_n) = 
 \frac{1}{\prod_{1 \leq i < j \leq n} (\alpha_i - \alpha_j)} 
 e^{-\sum_{i=1}^n \alpha_i y_i} \text{det}\left(D^{\alpha_1 \ldots \alpha_j}
   f_i(y_j) \right)_{i, j = 1}^n
\end{equation*}
and so completes the inductive step with distinct 
$(\alpha_1, \ldots, \alpha_n)$.

For general $(\alpha_1, \ldots, \alpha_n)$ such that $\alpha_i > 0$ 
for each $i =1, \ldots, n$ we prove the result by a continuity argument 
in $\alpha$. By Proposition \ref{time_change_refl} we have the following 
representation of the invariant measure:
\begin{equation*}
 (Y_1^*, \ldots, Y_n^*) \stackrel{d}{=} \left(\sup_{0 \leq s \leq \infty} Z_1^1(s), \ldots, \sup_{0 \leq s \leq \infty} Z_n^n(s)\right)
\end{equation*}
and in the proof we also showed that almost surely there exists some random time $v$ such that all of 
the suprema on the right hand side have stabilised. Moreover for any $\epsilon > 0$ 
this time can be chosen uniformly over
drifts bounded away from the origin $\alpha_1 \geq \epsilon \ldots, \alpha_n \geq \epsilon$.    We can construct a realisation of the Brownian paths 
$(B_1^{(-\alpha_1)}, \ldots, B_n^{(-\alpha_n)})$ so that they are continuous in $\alpha_1, \ldots, \alpha_n$ in the supremum norm on  compact time intervals. 
Therefore since $\epsilon$ is arbitrary we obtain that the right hand side is almost surely
continuous in the variables $(\alpha_1, 
\ldots, \alpha_n)$ on the set $(0, \infty)^n$. 
Therefore the distribution of $(Y_1^*, \ldots, Y_n^*)$ is continuous on the same set,
and so is the distribution of $(G(1, n), \ldots, G(1, 1))$ (as a finite number of 
operations of summation and maxima applied to exponential random variables).
This continuity completes the proof for any $\alpha_i > 0$ for $i = 1, \ldots, n$.
\end{proof}

\begin{proof}[Proof of Theorem \ref{sup_lpp}]
The Theorem follows by combining Theorem \ref{equality_law} with Proposition 
\ref{time_change_refl} part (ii).
\end{proof}

\section{Finite temperature}
\label{finite_temp_section}

\subsection{Time reversal}
\label{finite_temp_time_reversal}

The partition function for a $1+1$ dimensional directed point-to-point polymer 
in a Brownian environment (also known as the O'Connell-Yor polymer 
and studied in \cite{oconnell_yor, o_connell2012}) 
is the random variable,
\begin{equation*}
 Z_n(t) = \int_{0 = s_0 < \ldots < s_{n-1} < s_n = t} e^{\sum_{i=1}^n
B_i^{(-\alpha_{n-i+1})}(s_i) - B_i^{(-\alpha_{n-i+1})}(s_{i-1})}
 ds_1 \ldots ds_{n-1}.
\end{equation*}
We define a second random variable with an extra integral over $s_0$
and with the drifts reordered, 
\begin{equation}
\label{defnY_finite_tmp}
  Y_n(t) = \int_{0 < s_0 < \ldots < s_{n-1} < s_n = t} e^{\sum_{i=1}^n
B_i^{(-\alpha_i)}(s_i) - B_i^{(-\alpha_i)}(s_{i-1})}
 ds_0 \ldots ds_{n-1}.
\end{equation}
This is the partition function for a $1+1$ dimensional directed polymer 
in a Brownian environment 
with a \emph{flat initial condition}.
A change of variables shows that 
\begin{IEEEeqnarray*}{rCl}
 Y_n(t) 
& = & \int_{0 = u_0 < \ldots < u_{n} < t} e^{\sum_{i=1}^n
B_i^{(-\alpha_i)}(t - u_{n-i}) - B_i^{(-\alpha_i)}(t - u_{n-i+1})}
 du_1 \ldots du_{n}
\end{IEEEeqnarray*}
by letting $t - u_i = s_{n-i}$. By time reversal of Brownian motions,
$(B_{n-i+1}^{(-\alpha_{n-i+1})}(t) - B_{n-i +1}^{(-\alpha_{n-i+1})}
(t - s)_{s \geq 0}\stackrel{d}{=} (B_i^{(-\alpha_{n-i+1})}(s))_{s \geq 0}$,
we obtain, 
\begin{equation}
\label{time_reversal_eq}
Y_n(t) \stackrel{d}{=}  \int_{0 = u_0 < \ldots < u_{n} < t} e^{\sum_{i=1}^n
B_{n-i+1}^{(-\alpha_{i})}(u_{n-i+1}) - B_{n-i+1}^{(-\alpha_{i})}(u_{n-i})}
 du_1 \ldots du_{n}
 = \int_{0}^{t} Z_n(s) ds
\end{equation}
where the final equality follows by changing the index of summation from $i$ to $n-i+1$. 
As $t \rightarrow \infty$, the right hand side 
converges to $\int_{0}^{\infty} Z_n(s) ds$ and we now check that this is an almost 
surely finite random variable.
We consider the drifts and Brownian motions separately and bound the contribution from 
the Brownian motions.
For each $j = 1, \ldots, n$ let $\delta_j > 0$ and observe that there exists random constants $K_1, \ldots, K_n$ such that 
$B_1(s) \leq K_1 + \delta_1 s$ for all $s > 0$ and  
$\sup_{0 \leq s \leq t} B_j(t) - B_j(s) \leq K_j + \delta_j t$ for $t \geq 0$ and each $j = 2, \ldots, n$.
By choosing $\delta_1 + \ldots + \delta_n < \min_{1 \leq j \leq n} \alpha_{j}$ this shows
that the negative drifts dominate and the integral is almost surely finite. 
As a result the left hand side 
of (\ref{time_reversal_eq}) converges 
in distribution to a random variable which we denote $Y_n^*$ 
which satisfies 
\begin{equation}
\label{defnYStarfiniteTemp}
 Y_n^* \stackrel{d}{=} \int_0^{\infty} Z_n(s) ds.
\end{equation}

\subsection{Exponentially reflecting Brownian motions with a wall}
\label{finite_temp_sdes}
We extend \eqref{defnY_finite_tmp} to a definition of a
vector $(Y_1, \ldots, Y_n)$ as a functional of $n$ independent Brownian 
motions with drifts $(B_1^{(-\alpha_1)}, \ldots, B_n^{(-\alpha_n)})$ according to 
\begin{equation*}
 Y_k(t) = \int_{0 < s_0 < \ldots < s_{k-1} < s_k = t} e^{\sum_{i=1}^k
B_i^{(-\alpha_i)}(s_i) - B_i^{(-\alpha_i)}(s_{i-1})}
 ds_0 \ldots ds_{k-1} \text{ for } k = 1, \ldots, n.
\end{equation*}
The system $(Y_1, \ldots, Y_n)$ can be described by a system of SDEs.
Let $X_{j} = \log  \left(\frac{1}{2} Y_{j} \right)$ and observe that by 
It\^{o}'s formula,
\begin{gather}
\label{defnX}
 dX_{1}(t) = dB_1^{(-\alpha_1)}(t) + (e^{-X_{1}(t)}/2) dt \\
\label{defnX_2} 
 dX_{j}(t) = dB_{j}^{(-\alpha_{j})}(t) + 
 e^{-(X_{j}(t) - X_{j-1}(t))} dt \text{ for } j = 2, \ldots, n.
\end{gather}
We will call $X$ a system of \emph{exponentially reflecting} Brownian motions with a 
(soft) wall 
at the origin. 
We observe that $(Y_1, \ldots, Y_n)$ starts with each co-ordinate at zero and 
that each co-ordinate is strictly positive for all strictly positive times. 
This constructs an entrance law 
for the process $(X_{1}, \ldots, X_{n})$ from negative infinity. 
We will be interested in the invariant measure of this system which 
is related to log partition functions of the 
log-gamma polymer (see Theorem \ref{finite_tmp_invariant_array}).

To prove this we embed exponentially 
reflecting Brownian motions with a wall in a larger 
system of interacting Brownian motions indexed by 
a triangular array $(X_{ij}(t) : i + j \leq n+1, t \geq 0)$
with a unique invariant measure given by a whole field of 
log partition functions for the log-gamma polymer.
The Brownian system that we consider 
(see equation (\ref{defnXarray}) for a formal definition)
involves particles evolving according to 
independent Brownian motions with a drift term which 
depends on the neighbouring particles.
The interactions in the drift terms are one-sided and drawn as $\rightarrow$ 
or $\leadsto$ in Figure \ref{fig_Xarray}
where the particle 
at the point of the arrow has a drift depending on the particle (or wall) 
at the base of the arrow. There are two types of interaction: 
\begin{enumerate}[(i)]
 \item $\rightarrow$ is an exponential drift depending on the difference of the two particles.
This corresponds in a zero-temperature limit to particles which are instantaneously 
reflected in order to maintain an interlacing.
\item $\leadsto$ is a more unusual interaction and corresponds in a zero temperature limit to a weighted 
indicator function applied to the difference of the two particles. The effect of introducing this interaction  
is that the process $X_{ij}$ 
when started from its invariant measure and run in reverse time is given by the process 
where the direction of each interaction is reversed
(see Proposition \ref{time_symmetry}).
\end{enumerate}

\begin{figure}
\centering
 \begin{tikzpicture}[scale = 1.2]
 \node at (4, 4) {$X_{11}$};
 \draw[->] (3.3, 4) -- (3.7, 4);
  \draw[->] (4, 3.7) -- (4, 3.3);
 \node at (3, 4) {$X_{12}$};
 \draw[->] (2.3, 4) -- (2.7, 4);
   \draw[->] (4, 2.7) -- (4, 2.3);
 \node at (2, 4) {$X_{13}$};
 \draw[->] (1.3, 4) -- (1.7, 4);
   \draw[->] (4, 1.7) -- (4, 1.3);
 \node at (1, 4) {$X_{14}$};
 \draw[->] (3.3, 3) -- (3.7, 3);
   \draw[->] (3, 3.7) -- (3, 3.3);
 \node at (4, 3) {$X_{21}$};
 \draw[->] (2.3, 3) -- (2.7, 3);
   \draw[->] (3, 2.7) -- (3, 2.3);
 \node at (3, 3) {$X_{22}$};
 \draw[->] (3.3, 2) -- (3.7, 2);
 \node at (2, 3) {$X_{23}$};
   \draw[->] (2, 3.7) -- (2, 3.3);
 \node at (4, 2) {$X_{31}$};
 \node at (3, 2) {$X_{32}$};
 \node at (4, 1) {$X_{41}$};
 \draw (0, 4) -- (4, 0);
 \path[draw = black, ->, snake it]    (3.3, 1.7) -- (3.7,1.3);
 \path[draw = black, ->, snake it]    (2.3, 2.7) --  (2.7,2.3);
 \path[draw = black, ->, snake it]    (1.3,3.7) -- (1.7, 3.3);
 \path[draw = black, ->, snake it]    (3.3,2.7) -- (3.7, 2.3);
 \path[draw = black, ->, snake it]    (3.3,3.7) -- (3.7, 3.3);
 \path[draw = black, ->, snake it]    (2.3,3.7) -- (2.7, 3.3);
  \draw[->] (0.6, 3.6) -- (0.8, 3.8);
  \draw[->] (1.6, 2.6) -- (1.8, 2.8);
   \draw[->] (2.6, 1.6) -- (2.8, 1.8);
    \draw[->] (3.6, 0.6) -- (3.8, 0.8);
 \end{tikzpicture}
\caption{The interactions in the system $\{X_{ij} : i + j \leq n+1\}$.}
\label{fig_Xarray}
\end{figure}
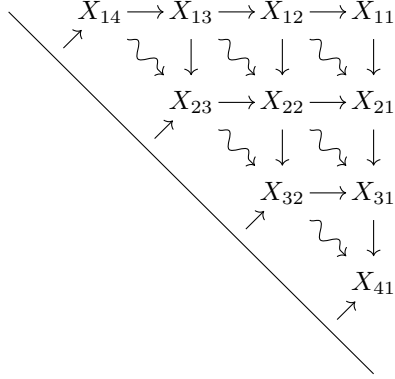

More formally we consider a diffusion process with values in $\mathbb{R}^{n(n+1)/2}$ whose 
generator is an operator $\mathcal{L}$
acting on functions $f \in C_c^{n(n+1)/2}(\mathbb{R})$ according to,
\begin{equation}
\label{defnXarray}
  \mathcal{L} f  = \sum_{\{(i, j):i+j \leq n+1\}}  \frac{1}{2} \frac{d^2 f}{dx_{ij}^2}  + b_{ij}(\mathbf{x}) \frac{df}{dx_{ij}} 
\end{equation}
where $\mathbf{x} = \{x_{ij}:i+j\leq n+1\}$ and
\begin{multline*}
 b_{ij}(\mathbf{x}) = -\alpha_{n-j+1} + \frac{(\alpha_{i-1}+\alpha_{n-j+1}) e^{x_{i j}}}{e^{x_{i-1, j +1}}+ e^{x_{ij}}} 1_{\{i > 1\}} 
 + e^{-(x_{ij} - x_{i, j+1})} 1_{\{i + j < n+1\}}\\
- e^{-(x_{i-1, j} - x_{i j})} 1_{\{i > 1\}} + \frac{1}{2}  e^{-x_{ij}} 1_{\{i + j = n+1\}}.
\end{multline*}


We observe that $\mathcal{L}$ restricted to functions of $(x_{1n}, \ldots, x_{11})$ 
alone is the generator for a system of exponentially reflecting Brownian motions 
with a wall, defined in
(\ref{defnX}, \ref{defnX_2}).

For foundational results on such a system we refer to 
Varadhan \cite{varadhan_lecture_notes} 
(see pages 197, 254, 259-260) which can be summarised in the following Lemma.
\begin{lemma}
\label{multidimensional_diffusions}
Let $L = \frac{1}{2} \Delta f + b\cdot \nabla$
where $b \in C^{\infty}(\mathbb{R}^d, \mathbb{R}^d)$.  
Suppose there exists a smooth function $u : \mathbb{R}^d \rightarrow (0, \infty)$ 
such that $u(x) \rightarrow \infty$ as $\lvert x \rvert \rightarrow \infty$ and 
$L u \leq cu$ for some $c > 0$.  
Then there exists a unique process with generator $L$ 
and the process does not explode.
Suppose furthermore there exists a 
smooth function $\phi$ such that $\phi \geq 0$, $\int_{\mathbb{R}^d} \phi = 1$ and $L^* \phi = 0$
where $L^*f = \frac{1}{2} \Delta f - \nabla \cdot (b f)$, 
then the measure with density $\phi$ is the unique invariant measure for the process
with generator $L$. 
\end{lemma}

\begin{lemma}
\label{non_exp_2}
Let $\mathcal{L}$ be the generator defined in  \eqref{defnXarray}.
There exists a smooth function $u : \mathbb{R}^d \rightarrow (0, \infty)$ 
such that $u(x) \rightarrow \infty$ as $\lvert x \rvert \rightarrow \infty$ and 
$\mathcal{L} u \leq cu$ for some $c > 0$. 
\end{lemma}
Therefore the conditions of Lemma \ref{multidimensional_diffusions}
are satisfied and there exists a unique process with generator $\mathcal{L}$ given by \eqref{defnXarray}
which 
does not explode.

\begin{proof}
 We define the function
 \begin{equation*}
  u(\mathbf{x}) = \sum_{\{(i,j) : i + j \leq n+1\}} e^{x_{ij}} + e^{-x_{ij}}
 \end{equation*}
which satisfies $u(\mathbf{x}) \rightarrow \infty$ as 
$\lvert \mathbf{x} \rvert \rightarrow \infty$. 
The diffusion terms and terms involving a bounded drift can all be easily bounded 
by a constant times $u$. We check this also holds for the terms 
involving unbounded drifts.
The terms involving a wall satisfy, 
\begin{equation*}
 e^{-x_{i, n-i+1}} \frac{du}{dx_{i, n - i+1}} = e^{-x_{i, n - i+1}} (e^{x_{i, n - i+1}}
 - e^{-x_{i, n-i+1}}) \leq 1.
\end{equation*}
The terms involving interlacing interactions between particles satisfy 
\begin{equation*}
 e^{-(x_{ij} - x_{i, j+1})} \frac{du}{x_{ij}}
 = e^{-(x_{ij} - x_{i, j+1})} (e^{x_{ij}} - e^{-x_{ij}}) \leq e^{x_{i, j+1}}
 \leq u(\mathbf{x})
\end{equation*}
and
\begin{equation*}
 -e^{-(x_{i-1, j} - x_{ij})} \frac{du}{dx_{ij}} = -e^{-(x_{i-1, j} - x_{i j})}
 (e^{x_{ij}} - e^{-x_{ij}}) 
 \leq e^{-x_{i-1, j}} \leq u(\mathbf{x}).
\end{equation*}
We sum over all interactions to prove that $u$ has the required properties.
\end{proof}

\subsection{The log-gamma polymer}

The invariant measure of both the exponentially reflecting Brownian motions with a wall defined in 
\eqref{defnX} and \eqref{defnX_2} and the $X$ array defined in 
\eqref{defnXarray} can be described by the 
log-gamma polymer. The log-gamma polymer originated in the work of  
Sepp{\"a}l{\"a}inen \cite{sep} and is defined as follows. Let 
$\{W_{i j}: (i, j) \in \mathbb{N}^2, i + j \leq n+1\}$ be a family of independent inverse
gamma 
random variables with densities,
\begin{equation}
\label{inv_gamma_density}
 P(W_{ij} \in dw_{ij}) = \frac{1}{\Gamma(\gamma_{i, j})}
 w^{-\gamma_{ij}} e^{-1/w_{ij}} \frac{dw_{ij}}{w_{ij}} \qquad \text{ for } w_{ij} > 0
\end{equation}
and parameters 
$\gamma_{ij} = \alpha_i + \alpha_{n - j+1}$. 
Let  $\Pi_n^{\text{flat}}(k, l)$
denote the set of all directed (up and right) paths from the point $(k, l)$ 
to the line $\{(i, j) : i + j = n+1\}$ and define the partition functions 
and log partition functions:
\begin{equation}
\label{defnZeta}
 \zeta_{kl} = \sum_{\pi \in \Pi_n^{\text{flat}}(k, l)} \prod_{(i,j) \in \pi} W_{ij}, 
 \qquad \qquad \qquad
 \xi_{kl} = \log \zeta_{kl}.
\end{equation}
These are the partition functions for a $(1+1)$ dimensional directed polymer in 
a random environment given by $\{W_{i j}: (i, j) \in \mathbb{N}^2, i + j \leq n+1\}$. 


\begin{lemma}
\label{log_gamma_explicit_density}
The distribution of $\xi_{ij}$ given $\xi_{i+1, j} = x_{i+1, j}$ and 
$\xi_{i, j+1} = x_{i, j+1}$ has a 
density with respect to Lebesgue measure proportional to
\begin{equation*}
 \exp\left(-(\alpha_i + \alpha_{n-j+1})x_{ij} - 
 e^{x_{i, j+1} -x_{ij}}
 -e^{x_{i+1, j} -x_{ij}} + (\alpha_i + \alpha_{n-j+1}) \log (e^{x_{i, j+1}} + 
 e^{x_{i+1, j}})\right) .
\end{equation*}
The distribution of the field $(\xi_{i,j} : i + j \leq n +1)$ has a density 
with respect to Lebesgue measure on $\mathbb{R}^{n(n+1)/2}$ proportional to
\begin{IEEEeqnarray*}{rCl}
 \pi(\mathbf{x})& = & \prod_{i + j < n+1} 
 \exp\bigg(-(\alpha_i + \alpha_{n-j+1})x_{ij} - 
 e^{x_{i, j+1} -x_{ij}}
 -e^{x_{i+1, j} -x_{ij}} \\
 && \qquad +  (\alpha_i + \alpha_{n-j+1}) \log (e^{x_{i, j+1}} + 
 e^{x_{i+1, j}})\bigg) 
\cdot \prod_{i=1}^n \exp\left(-2\alpha_{i} x_{i, n - i+1} -
 e^{-x_{i, n - i+1}}\right).
\end{IEEEeqnarray*}
\end{lemma}

\begin{proof}
The partition functions satisfy a local update rule
$\zeta_{ij} = (\zeta_{i, j+1} + \zeta_{i+1, j}) W_{ij}$ and equivalently 
$\xi_{ij} = \log W_{ij} + \log (e^{\xi_{i, j+1}} + e^{\xi_{i+1, j}})$. 
This combined with the explicit density for the inverse gamma density (\ref{inv_gamma_density})
proves the first statement. 
The second part then follows 
by an iterative application of the first part.
\end{proof}

\subsection{The invariant measure of exponentially reflecting Brownian 
motions with a wall and the log-gamma polymer}

\begin{theorem}
\label{finite_tmp_invariant_array}
Let $(X_{ij}(t) : i + j \leq n+1, t \geq 0)$ be the diffusion with 
generator (\ref{defnXarray}). This has a unique invariant measure 
which we denote $(X_{ij}^* : i + j \leq n+1)$ and satisfies
\begin{equation*}
      (X_{ij}^*: i + j \leq n+1) \stackrel{d}{=} (\xi_{i j}: i + j \leq n+1).
     \end{equation*}
A consequence is that $(\xi_{1 n}, \ldots, \xi_{1 1})$ is distributed as the unique 
invariant measure of the system of
exponentially reflecting Brownian motions with a wall, defined in (\ref{defnX}, 
\ref{defnX_2}).
\end{theorem}

A key role in the proof will be played by inductive decompositions of the generator for the Brownian 
system in 
\eqref{defnXarray} and the explicit density for 
the log-gamma polymer in Lemma \ref{log_gamma_explicit_density}.   
Let $S \subset \mathbb{N}^2 \cap \{(i, j) : i + j \leq n+1\}$ have 
a boundary given by a down-right path in 
the orientation of Figure \ref{update_figure} (the boundary 
is denoted by the dotted line) -- explicitly we require that if
$(i, j) \in S$ then $(i+k, j+l)\in S$ for all $k, l \geq 0$ 
such that $i + j + k + l \leq n+1$.
We can define the log-gamma polymer on $S$ and we denote the density of log partition 
functions on $S$ by $\pi_S(x)$. 
Proposition \ref{log_gamma_explicit_density} proves that
$\pi_S(\mathbf{x})$ is proportional to $\exp(-V_S(\mathbf{x}))\prod_{(i, j) \in S}dx_{ij}$ with 
\begin{IEEEeqnarray*}{rCl}
 V_S(\mathbf{x})& = & \sum_{(i, j) \in S \setminus D_n} 
 \bigg( (\alpha_i + \alpha_{n-j+1})x_{ij} + 
 e^{x_{i, j+1} -x_{ij}}
 + e^{x_{i+1, j} -x_{ij}}   \\
&& \qquad \qquad - (\alpha_i + \alpha_{n-j+1}) \log (e^{x_{i, j+1}} + 
 e^{x_{i+1, j}})\bigg) + \sum_{i=1}^n (2\alpha_{i} x_{i, n - i+1} +
 e^{-x_{i, n - i+1}})
\end{IEEEeqnarray*}
where $D_n = \{(i, j) \in \mathbb{N}^2: i + j = n+1\}$.
We can build the density of the log-gamma polymer inductively by adding an extra vertex $(i, j)$ to $S$ 
and assuming that both $S$ and $S \cup (i, j)$ have down-right boundaries in the orientation 
of Figure \ref{update_figure}. We observe that
$V_{S \cup \{i, j\}} = V_S + V^*$ where
\begin{equation}
\label{Vstar}
V^* =   (\alpha_i + \alpha_{n-j+1}) x_{ij}
 + e^{-(x_{ij} - x_{i+1, j})} + e^{-(x_{ij} - x_{i, j+1})}
 - (\alpha_i + \alpha_{n-j+1}) \log (e^{x_{i, j +1}} + e^{x_{i+1, j}}).
\end{equation}

We now consider an inductive decomposition of the generator in \eqref{defnXarray} which 
is related to the above decompsoition of the log-gamma polymer.
We consider a Brownian system with particles indexed by $S$
which (i) agrees with the process with generator $\mathcal{L}$ when 
$S = \{(i, j) :i+j \leq n+1\}$ and (ii) has an invariant measure with density $\pi_S$. 
The process can be represented by the interactions present in the diagram on the left hand
side of Figure \ref{update_figure}. 
We consider a diffusion with values indexed by $S$ with generator $\mathcal{L}_S$, 
acting on functions $f \in C_c^{n(n+1)/2}(\mathbb{R})$ as follows,
\begin{IEEEeqnarray}{rCl}
 \mathcal{L}_S f & = &  \sum_{(i, j) \in S \setminus D_n} \bigg(\frac{1}{2} \frac{d^2 f}{dx_{ij}^2} - \alpha_{n-j+1}  \frac{df}{dx_{ij}}
 + e^{-(x_{ij} - x_{i, j+1})} \frac{d}{dx_{ij}} - e^{-(x_{ij} - x_{i+1, j})} \frac{d}{dx_{i+1, j}} 
  \IEEEnonumber  \\
  & & 
 \quad + 
 \frac{(\alpha_i + \alpha_{n-j+1}) e^{x_{i+1, j}}}{e^{x_{i+1, j}} + e^{x_{i, j+1}}} \frac{d}{dx_{i+1, j}} \bigg)
   + \sum_{(i, j) \in S \cap D_n} \frac{1}{2} \frac{d^2 f}{dx_{ij}^2} - \alpha_{n-j+1}  \frac{df}{dx_{ij}} 
   + \frac{1}{2}e^{-{x_{ij}}} \frac{df}{dx_{ij}}.
\end{IEEEeqnarray}
For the same class of sets $S$ we consider a second diffusion with generator $\mathcal{A}_S$, 
acting on functions $f \in C_c^{n(n+1)/2}(\mathbb{R})$ as follows,
\begin{IEEEeqnarray}{rCl}
\label{defnA}
 \mathcal{A}_S f & = &  \sum_{(i, j) \in S \setminus D_n} \bigg( \frac{1}{2} \frac{d^2 f}{dx_{ij}^2} - \alpha_{i}  \frac{df}{dx_{ij}}
 + e^{-(x_{ij} - x_{i+1, j})} \frac{d}{dx_{ij}} - e^{-(x_{ij} - x_{i, j+1})} \frac{d}{dx_{i, j+1}} 
  \IEEEnonumber  \\
  & & 
 \quad + 
 \frac{(\alpha_i + \alpha_{n-j+1}) e^{x_{i, j+1}}}{e^{x_{i+1, j}} + e^{x_{i, j+1}}} \frac{d}{dx_{i, j+1}} \bigg)
   + \sum_{(i, j) \in S \cap D_n} \frac{1}{2} \frac{d^2 f}{dx_{ij}^2} - \alpha_{i}  \frac{df}{dx_{ij}} 
   + \frac{1}{2} e^{-{x_{ij}}} \frac{df}{dx_{ij}}.
\end{IEEEeqnarray}

Proposition \ref{multidimensional_diffusions} and Lemma \ref{non_exp_2}
show that there exists unique processes with generators $\mathcal{L}_S$ and 
$\mathcal{A}_S$ and that these processes do not explode. 
The motivation for considering $\mathcal{A}_S$ is 
that the process with this generator will be the 
time reversal of the process with generator $\mathcal{L}_S$ when the process is 
run in its invariant measure $\pi_S$.
The process with operator  $\mathcal{A}_S$
can be represented by a diagram in the same way as $\mathcal{L}_S$ 
in Figure \ref{fig_Xarray},
where for the $\mathcal{A}_S$ process the direction of every 
interaction is reversed.

\begin{figure}
\centering
 \begin{tikzpicture}[scale = 1,   tlabel/.style={pos=0.4,right=-1pt},
 baseline=(current bounding box.center)]
 \filldraw (3,4) circle (1pt);
 \filldraw (2,4) circle (1pt);
 \filldraw (3,3) circle (1pt);
 \filldraw (4,2) circle (1pt);
 \filldraw (3,2) circle (1pt);
 \filldraw (4,1) circle (1pt);
 \filldraw (1,4) circle (1pt);
 \draw[->] (2.3, 4) -- (2.7, 4);
 \draw[->] (1.3, 4) -- (1.7, 4);
   \draw[->] (4, 1.7) -- (4, 1.3);
   \draw[->] (3, 3.7) -- (3, 3.3);
 \draw[->] (2.3, 3) -- (2.7, 3);
   \draw[->] (3, 2.7) -- (3, 2.3);
 \draw[->] (3.3, 2) -- (3.7, 2);
   \draw[->] (2, 3.7) -- (2, 3.3);
 \draw (0, 4) -- (4, 0);
  \path[draw = black, ->, snake it]    (3.3, 1.7) -- (3.7,1.3);
 \path[draw = black, ->, snake it]    (2.3, 2.7) --  (2.7,2.3);
 \path[draw = black, ->, snake it]    (1.3,3.7) -- (1.7, 3.3);
 \path[draw = black, ->, snake it]    (2.3,3.7) -- (2.7, 3.3);
  \draw[->] (0.6, 3.6) -- (0.8, 3.8);
  \draw[->] (1.6, 2.6) -- (1.8, 2.8);
   \draw[->] (2.6, 1.6) -- (2.8, 1.8);
    \draw[->] (3.6, 0.6) -- (3.8, 0.8);
 \draw[dotted](0, 4.2) -- (3.2, 4.2) -- (3.2, 2.2) -- (4.2, 2.2) -- (4.2, 0); 
\end{tikzpicture}
\hspace{1cm}$\Huge\rightarrow$
\begin{tikzpicture}[scale = 1,  tlabel/.style={pos=0.4,right=-1pt},  baseline=(current bounding box.center)]
 \filldraw (3,4) circle (1pt);
 \filldraw (2,4) circle (1pt);
 \filldraw (3,3) circle (1pt);
 \filldraw (4,2) circle (1pt);
 \filldraw (3,2) circle (1pt);
 \filldraw (4,1) circle (1pt);
 \filldraw (1,4) circle (1pt);
 \draw[->] (2.3, 4) -- (2.7, 4);
 \filldraw (4, 3) circle (1pt);
 \draw[->] (3.3, 3) -- (3.7, 3);
  \draw[->] (4, 2.7) -- (4, 2.3);
    \path[draw = black, ->, snake it]    (3.3, 2.7) -- (3.7,2.3);
 \draw[->] (1.3, 4) -- (1.7, 4);
   \draw[->] (4, 1.7) -- (4, 1.3);
   \draw[->] (3, 3.7) -- (3, 3.3);
 \draw[->] (2.3, 3) -- (2.7, 3);
   \draw[->] (3, 2.7) -- (3, 2.3);
 \draw[->] (3.3, 2) -- (3.7, 2);
   \draw[->] (2, 3.7) -- (2, 3.3);
 \draw (0, 4) -- (4, 0);
  \path[draw = black, ->, snake it]    (3.3, 1.7) -- (3.7,1.3);
 \path[draw = black, ->, snake it]    (2.3, 2.7) --  (2.7,2.3);
 \path[draw = black, ->, snake it]    (1.3,3.7) -- (1.7, 3.3);
 \path[draw = black, ->, snake it]    (2.3,3.7) -- (2.7, 3.3);
  \draw[->] (0.6, 3.6) -- (0.8, 3.8);
  \draw[->] (1.6, 2.6) -- (1.8, 2.8);
   \draw[->] (2.6, 1.6) -- (2.8, 1.8);
    \draw[->] (3.6, 0.6) -- (3.8, 0.8);
 \draw[dotted](0, 4.2) -- (3.2, 4.2) -- (3.2, 3.2) -- (4.2, 3.2) -- (4.2, 0); 
 \end{tikzpicture}
  \caption{Updating $\mathcal{L}_S$ to $\mathcal{L}_{S \cup \{(i, j)\}}$.}
 \label{update_figure}
\end{figure}
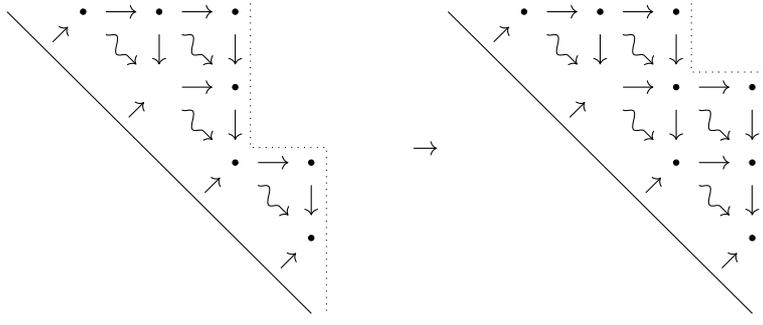
We add in a vertex $(i, j)$ as described in Figure \ref{update_figure}, where we assume that both $S$ and 
$S \cup (i, j)$ have boundaries with down-right paths in the orientation of Figure \ref{update_figure}. 
Then, 
\begin{multline}
\label{decompose_Ls}
 \mathcal{L}_{S\cup\{(i, j)\}} = \mathcal{L}_{S} + \frac{1}{2}\frac{d^2}{d^2 {x_{ij}}} 
 - \alpha_{n-j+1}\frac{d}{dx_{ij}}
 + e^{-(x_{ij}-x_{i,j+1})}\frac{d}{dx_{ij}} \\ - 
 e^{-(x_{ij} - x_{i+1, j})}\frac{d}{dx_{i+1, j}}  +
 \frac{(\alpha_i + \alpha_{n-j+1})e^{x_{i+1, j}}}{e^{x_{i+1, j}} + 
 e^{x_{i, j+1}}} \frac{d}{dx_{i+1, j}}
\end{multline}
and
\begin{multline}
\label{decompose_As}
 \mathcal{A}_{S \cup \{i, j\}} = \mathcal{A}_{S} + 
  \frac{1}{2}\frac{d^2}{d^2 {x_{ij}}} - \alpha_i \frac{d}{dx_{ij}} +
 e^{-(x_{ij}-x_{i+1,j})}\frac{d}{dx_{ij}} \\ - 
 e^{-(x_{ij} - x_{i, j+1})}\frac{d}{dx_{i, j+1}} +
 \frac{(\alpha_i + \alpha_{n-j+1})e^{x_{i, j+1}}}{e^{x_{i+1, j}} + 
 e^{x_{i, j+1}}} \frac{d}{dx_{i, j+1}}.
\end{multline}

\begin{lemma}
\label{gradient_diffusion_lemma}
For any subset $S$ with a down-right boundary in the orientation of Figure \ref{update_figure}, 
the diffusion with generator $\frac{1}{2}(\mathcal{L}_S + \mathcal{A}_S)$ is a gradient diffusion satisfying 
 \begin{equation*}
\frac{1}{2}(\mathcal{L}_S + \mathcal{A}_S) = \frac{1}{2}\Delta_S -\frac{1}{2}\nabla V_S \cdot \nabla_S,   
 \end{equation*}
where $\Delta_S = \sum_{ij \in S} \frac{d^2}{dx_{ij}^2}$ and $\nabla_S = \sum_{ij \in S} \frac{d}{dx_{ij}}$. 
In particular, the process 
with generator $\frac{1}{2}(\mathcal{L}_S + \mathcal{A}_S)$ has
invariant measure given by $\pi_S$ and is reversible when run in its invariant measure.
\end{lemma}

\begin{proof}
We use the inductive decompositions of $\mathcal{L}, \mathcal{A}$ and $V$ 
to check the Lemma 
inductively. 
For the base case we let $S = \{(i, j): i+j= n+1\}$ and observe that in this case $\mathcal{L}_S = \mathcal{A}_S$ 
and both are the generators for
$n$ independent exponentially reflecting Brownian motions with a wall. 
Then the Lemma follows from:
\begin{equation*}
-\frac{1}{2}\frac{dV_S}{dx_{i, n-i+1}} =  \alpha_i + \frac{1}{2} e^{-(x_{i, n-i+1})}. 
\end{equation*}
For the inductive step we consider a set $S$ with a down-right boundary and add 
an extra vertex $(i, j)$ with the property that $S \cup (i, j)$ also has a  
down-right boundary. We show that
\begin{equation}
\label{gradient_diffusion_check}
\frac{d^2}{dx_{ij}^2} - \nabla V^* \cdot \nabla_{S \cup {(i, j)}}
= \mathcal{L}_{S \cup \{i, j\}} - \mathcal{L}_S 
+  \mathcal{A}_{S \cup \{i, j\}} - \mathcal{A}_S,
\end{equation}
by calculating the non-zero co-ordinates of $\nabla V^*$:
\begin{gather*}
 \frac{dV^*}{dx_{i,j+1}} = e^{-(x_{ij} - x_{i,j+1})} - \frac{(\alpha_i + \alpha_{n-j+1})e^{x_{i, j+1}}}{e^{x_{i,j+1}} + e^{x_{i+1, j}}}, \\ 
 \frac{dV^*}{dx_{i+1, j}} = e^{-(x_{ij} - x_{i+1, j})} - \frac{(\alpha_i + \alpha_{n-j+1})e^{x_{i+1, j}}}{e^{x_{i,j+1}} + e^{x_{i+1, j}}} \\
 \frac{dV^*}{dx_{ij}} = \alpha_i + \alpha_{n-j+1} -e^{-(x_{ij} - x_{i,j+1})}-e^{-(x_{ij} - x_{i+1, j})}
\end{gather*}
and observing that this gives equality with the right hand side of \eqref{gradient_diffusion_check} by 
using equations \eqref{decompose_Ls} and \eqref{decompose_As}.
\end{proof}

\begin{lemma}
\label{drift_potential_orthogonality}
 Let $S$ be a subset $S$ with a down-right boundary in the orientation of Figure \ref{update_figure} and
let $d_{S}$ denote the difference in drifts between $\mathcal{L}_{S}$ 
and $\mathcal{A}_{S}$. Then
\begin{enumerate}[(i)]
 \item The vector field $d_S$ is divergence-free, \begin{equation*}
        \nabla \cdot d_S = 0
       \end{equation*}
 \item The vector fields $d_S$ and $\nabla V_S$ are orthogonal, \begin{equation*}
 (d_S, \nabla V_S) = 0
\end{equation*}
\end{enumerate}
\end{lemma}


\begin{proof}
We prove both parts inductively. For the base case we let $S = \{(i, j): i+j= n+1\}$
and observe that $\nabla \cdot d_S = 0$ and $(d_S, \nabla V_S) = 0$ both hold because $d_S = 0$. 
For any set $S$ with a down-right boundary, we add in a new vertex $(i, j)$ with the property 
that $S \cup (i, j)$ also has 
a down-right boundary.
For part (i), the difference of drifts inherits an inductive decomposition from
$\mathcal{L}_S$ and $\mathcal{A}_S$:
\begin{equation*}
 d_{S \cup \{(i, j)\}} = d_S + d^* 
\end{equation*}
where $d_S$ is extended to be $\mathbb{R}^{S \cup \{i, j\}}$ valued by setting $d_S(i, j) = 0$.
Every component of $d^*$ is zero except for the following:
\begin{gather}
\label{dstar_1}
 d^*(i, j+1) = e^{-(x_{ij} - x_{i, j+1})} - 
  \frac{(\alpha_i + \alpha_{n-j+1})e^{x_{i, j+1}}}{e^{x_{i, j+1}} + e^{x_{i+1, j}}} \\
  \label{dstar_2}
d^*(i+1, j) = -e^{-(x_{ij} - x_{i+1, j})} + 
  \frac{(\alpha_i + \alpha_{n-j+1})e^{x_{i+1, j}}}{e^{x_{i, j+1}} + e^{x_{i+1, j}}} \\
  d^*(i, j) = 
  \alpha_i - \alpha_{n-j+1} + e^{-(x_{ij} - x_{i, j+1})} - e^{-(x_{ij} - x_{i+1, j})}.
  \label{dstar_3}
\end{gather}
We observe that 
\begin{equation*}
 \nabla \cdot d^* = 0
\end{equation*}
by differentiating (\ref{dstar_1}-\ref{dstar_3}) to obtain the following, 
\begin{gather*}
 \frac{d}{dx_{i, j+1}}  d^* = e^{-(x_{ij}-x_{i, j+1})} + \frac{(\alpha_i + \alpha_{n-j+1})e^{2 x_{i, j+1}}}{(e^{x_{i+1, j}} + e^{x_{i, j+1}})^2}
 -\frac{(\alpha_i + \alpha_{n-j+1})e^{x_{i, j+1}}}{e^{x_{i+1, j}} + e^{x_{i, j+1}}}\\
 \frac{d}{dx_{i+1, j}}  d^* = - e^{-(x_{ij} - x_{i+1, j})} - \frac{(\alpha_i + \alpha_{n-j+1})e^{2 x_{i+1, j}}}{(e^{x_{i+1, j}} + e^{x_{i, j+1}})^2} 
 + \frac{(\alpha_i + \alpha_{n-j+1})e^{x_{i+1, j}}}{e^{x_{i+1, j}} + e^{x_{i, j+1}}}\\
 \frac{d}{dx_{i j}}  d^* = -e^{-(x_{ij} - x_{i, j+1})} + e^{-(x_{ij} - x_{i+1, j})},
\end{gather*}
and observing that the sum equals zero.
Combining this with the inductive hypothesis, that $\nabla \cdot d_S = 0$, shows that $\nabla \cdot d_{S \cup (i, j)} = 0$.

For part (ii), we assume the
inductive hypothesis, that $(d_S, \nabla V_S) =0$, 
and observe that this means $(d_{S \cup (i,j)}, \nabla V_{S \cup (i,j)}) = 0$ is 
equivalent to the following identity: 
 \begin{equation}
\label{drift_potential_identity}
(d^*, \nabla V_S) + (d_S, \nabla V^*) + (d^*, \nabla V^*) = 0.
\end{equation}

We observe that $d^{*}$ and $\nabla V^*$ are only 
non-zero in the co-ordinates $(i, j+1), (i+1, j)$ and $(i,j)$ so we can restrict 
to considering $\nabla V_S$ and $d_S$ in these coordinates.

We observe that by definition
$d_S(i, j) = 0$ and 
\begin{IEEEeqnarray}{rCl}
\label{dS_1}
 d_S(i, j+1) &  = & \alpha_i - \alpha_{n-j} + e^{-(x_{i, j+1} - x_{i, j+2})}1_{\{i + j < n\}}
 - e^{-(x_{i, j+1} - x_{i+1, j+1})}1_{\{i + j < n\}} \IEEEnonumber\\
 & & \quad + \frac{(\alpha_{n-j}+\alpha_{i-1})e^{x_{i,j+1}}}{e^{x_{i,j+1}} + e^{x_{i-1, j+2}}} 1_{\{(i-1, j+1) \in S\}}
 - e^{-(x_{i-1, j+1} - x_{i, j+1})} 1_{\{(i-1, j+1) \in S\}} \\
  d_S(i+1, j) & = & \alpha_{i+1} - \alpha_{n-j+1} + e^{-(x_{i+1, j} - x_{i+1, j+1})}1_{\{i+j < n\}}
 - e^{-(x_{i+1, j} - x_{i+2, j})}1_{\{i+j < n\}} \IEEEnonumber\\
 & & \quad -\frac{(\alpha_{n-j+2} + \alpha_{i+1})e^{x_{i+1, j}}}{(e^{x_{i+1, j}} + e^{x_{i+2, j-1}})} 1_{\{(i+1, j-1) \in S\}}
 + e^{-(x_{i+1, j-1} - x_{i+1, j})} 1_{\{(i+1, j-1) \in S\}}.
 \label{dS_2}
\end{IEEEeqnarray}
The indicator functions correspond to the effect of $\leadsto$ and $\downarrow$ 
or $\rightarrow$ interactions which may or 
may not be present depending on the shape of $S$. 
We also note that for $i + j = n$, then we have $\alpha_i - \alpha_{n-j} =  \alpha_{i+1} - \alpha_{n-j+1} = 0$.


For $i + j < n$, 
the terms in $V_S$ which involve any of $x_{i, j+1}, x_{i+1, j}$ or $x_{ij}$ are given via
the following decompositions: 
\begin{equation}\begin{split}
   \MoveEqLeft
   V_S  = 
   (\alpha_{i} + \alpha_{n-j}) x_{i j+1} + e^{-(x_{i j+1} - x_{i j+2})} 
 + e^{-(x_{i j+1} - x_{i+1 j+1})} 
 + (\alpha_{i+1} + \alpha_{n-j+1})x_{i+1 j} \\
  & + e^{-(x_{i+1 j} - x_{i+1 j+1})} 
 + e^{-(x_{i+1 j} - x_{i+2 j})} 
 + e^{-(x_{i-1 j+1} - x_{i j+1})} 1_{\{(i-1, j+1) \in S\}} \\
 & + e^{-(x_{i+1 j-1} - x_{i+1 j})} 1_{\{(i+1, j-1) \in S\}}  
 - (\alpha_{i-1} + \alpha_{n-j})\log(e^{x_{i j+1}} + e^{x_{i-1, j+2}}) 1_{\{(i-1, j+1) \in S\}} \\
 &   - (\alpha_{i+1} + \alpha_{n-j+2}) \log (e^{x_{i+1 j}} + e^{x_{i+2, j-1}})1_{\{(i+1, j-1) \in S\}} + \tilde{V}_S
 \end{split}
    \label{VS_1}
 \end{equation}
 where $\tilde{V}_S$ does not depend on any of: $x_{i j+1},$ $x_{i+1 j}$, or $x_{ij}$.
For $i + j = n$,
\begin{equation}\begin{split}
 \MoveEqLeft  V_S  = 2
\alpha_{i} x_{i, j+1} + e^{-x_{i, j+1}}+
2\alpha_{i+1} x_{i+1, j} + e^{-x_{i+1, j}} +   e^{-(x_{i-1 j+1} - x_{i j+1})} 1_{\{(i-1, j+1) \in S\}} \\
 & + e^{-(x_{i+1 j-1} - x_{i+1 j})} 1_{\{(i+1, j-1) \in S\}} 
 - (\alpha_{i-1} + \alpha_{n-j})\log(e^{x_{i j+1}} + e^{x_{i-1, j+2}}) 1_{\{(i-1, j+1) \in S\}}  \\
   & - (\alpha_{i+1} + \alpha_{n-j+2}) \log (e^{x_{i+1 j}} + e^{x_{i+2, j-1}})1_{\{(i+1, j-1) \in S\}} + \tilde{V}_S
 \end{split}
  \label{VS_2}
 \end{equation}
where $\tilde{V}_S$ does not depend on any of: $x_{i j+1},$ $x_{i+1 j}$, or $x_{ij}$.

Therefore we will check  (\ref{drift_potential_identity}) by using equations (\ref{Vstar}, \ref{dstar_1}-\ref{dstar_3}, \ref{dS_1}-\ref{dS_2},
\ref{VS_1}-\ref{VS_2}) in the following.
We will first observe that the terms involving indicator functions vanish. 
The terms in $\nabla V_S (i, j+1)$ involving $1_{\{(i-1, j+1) \in S\}}$ are equal to
\begin{equation*}
 \left( e^{-(x_{i-1 j+1} - x_{i j+1})} - \frac{(\alpha_{i-1} + \alpha_{n-j})e^{x_{i j+1}}}{e^{x_{i j+1}} + e^{x_{i-1 j+2}}}\right) 1_{\{(i-1, j+1) \in S\}}.
\end{equation*}
This is the negative of the terms in $d_S(i, j+1)$ involving $1_{\{(i-1, j+1) \in S\}}$  from 
\eqref{dS_1}. We have shown above 
that $\nabla V^* (i, j+1) = d^*(i, j+1)$.
Therefore the terms involving 
indicator functions $1_{\{(i-1, j+1) \in S\}}$ 
cancel in the sum $(d^*, \nabla V_S) + (d_S, \nabla V^*)$. The terms involving $1_{\{(i+1, j-1) \in S\}}$ also
cancel in the sum $(d^*, \nabla V_S) + (d_S, \nabla V^*)$. In this case, $\nabla V^*(i+1, j) = - d^*(i+1, j)$ and the terms involving $1_{\{(i+1, j-1) \in S\}}$ in $\nabla V_S(i+1, j)$ and 
$d_S(i+1, j)$ are equal. 

Therefore it is sufficient to show that equation \eqref{drift_potential_identity} holds 
in the case when neither $(i-1, j+1)$ nor $(i+1, j-1)$ are in $S$.
This is a useful simplification 
and we calculate in this case for $i+j < n$,
\begin{equation*}
 \begin{split}
\MoveEqLeft
 (d^*, \nabla V_S) = \bigg(e^{-(x_{ij} - x_{i j+1})} - 
 \frac{(\alpha_i + \alpha_{n-j+1})e^{x_{i j+1}}}{e^{x_{i j+1}} + e^{x_{i+1 j}}}\bigg)\bigg(\alpha_i + \alpha_{n-j} - 
 e^{-(x_{i j+1} - x_{i j+2})}-e^{-(x_{i j+1} - x_{i+1 j+1)}}\bigg)\\
& +  \bigg(-e^{-(x_{ij} - x_{i+1 j})} + 
 \frac{(\alpha_i + \alpha_{n-j+1})e^{x_{i+1 j}}}{e^{x_{i j+1}} + e^{x_{i+1 j}}}\bigg)\bigg(\alpha_{i+1} + \alpha_{n-j+1} - 
 e^{-(x_{i+1 j} - x_{i+2 j})} -e^{-(x_{i+1 j} - x_{i+1 j+1})}\bigg)  
 \end{split}
\end{equation*}
\begin{equation*}
 \begin{split}
\MoveEqLeft
 (d_S, \nabla V^*) = \bigg(\alpha_i - \alpha_{n-j} + e^{-(x_{i j+1} - x_{i j+2})} - e^{-(x_{i j+1} - x_{i+1 j+1})}\bigg) 
 \bigg( e^{-(x_{ij}-x_{i j+1})} - \frac{(\alpha_i + \alpha_{n-j+1})e^{x_{i j+1}}}{e^{x_{i j+1}} + e^{x_{i+1 j}}}\bigg) \\
 & + \bigg(\alpha_{i+1} - \alpha_{n-j+1} + e^{-(x_{i+1 j} - x_{i+1 j+1})} - e^{-(x_{i+1 j} - x_{i+2 j})}\bigg) 
 \bigg(e^{-(x_{ij}-x_{i+1 j})} - \frac{(\alpha_i + \alpha_{n-j+1})e^{x_{i+1 j}}}{e^{x_{i j+1}} + e^{x_{i+1 j}}}\bigg)  
 \end{split}
\end{equation*}
\begin{equation*}
\begin{split}
\MoveEqLeft
 (d^*, \nabla V^*) = \bigg(e^{-(x_{ij}-x_{i j+1})} - \frac{(\alpha_i + \alpha_{n-j+1})e^{x_{ij+1}}}{e^{x_{i j+1}}+e^{x_{i+1 j}}}\bigg)
 \bigg(e^{-(x_{ij} - x_{i j+1})}- \frac{(\alpha_i + \alpha_{n-j+1})e^{x_{i j+1}}}{e^{x_{i+1 j}} + e^{x_{i j+1}}}\bigg)\\
 & + \bigg(-e^{-(x_{ij}-x_{i+1 j})} + \frac{(\alpha_i + \alpha_{n-j+1})e^{x_{i+1 j}}}{e^{x_{i+1 j}} + e^{x_{i j+1}}}\bigg)
 \bigg(e^{-(x_{ij} - x_{i+1 j})} - \frac{(\alpha_i + \alpha_{n-j+1})e^{x_{i+1 j}}}{e^{x_{i+1 j}} + e^{x_{i j+1}}} \bigg)\\
 & + \bigg( \alpha_i - \alpha_{n-j+1} + e^{-(x_{ij} - x_{i j+1})} - e^{-(x_{ij} - x_{i+1 j})}\bigg) 
 \bigg(\alpha_i + \alpha_{n-j+1} - e^{-(x_{ij} - x_{i j+1})} - e^{-(x_{ij} - x_{i+1 j})}\bigg).
 \end{split}
\end{equation*}
The following (non-obvious) cancellation then proves that equation \eqref{drift_potential_identity} holds.
For $i + j < n$, it is easy to see that all terms involving $e^{-(x_{ij+1} - x_{i j+2})}$ cancel 
and this similarly holds for the terms $e^{-(x_{i+1 j} - x_{i+2 j})}$. It is useful to consider 
all terms that involve either $e^{-(x_{i j+1} - x_{i+1 j+1})}$ 
or $e^{-(x_{i+1 j} - x_{i+1 j+1})}$ together and all such terms cancel. In the case $i+j = n$, none of these terms are present, however, 
there is an extra $-e^{x_{i j+1}}-e^{x_{i+1 j}}$ in $\nabla V_S$ which cancels in $(d^*, \nabla V_S)$. 
The remaining calculation for the cases $i + j < n$ and $i + j = n$ is the same.

Once these cancellations have been performed the left hand side of \eqref{drift_potential_identity}
is a function of $x_{i j+1}, x_{i+1 j}$ and $x_{ij}$ alone, and has a much simpler form.  
In particular, after this cancellation $ (d^*, \nabla V_S) + (d_S, \nabla V^*)$ equals
\begin{equation*}
2\alpha_i\left(e^{-(x_{ij}-x_{i j+1})} - \frac{(\alpha_i + \alpha_{n-j+1})e^{x_{i j+1}}}{e^{x_{i j+1}} + e^{x_{i+1 j}}}\right) +
 2\alpha_{n-j+1} \left(-e^{-(x_{ij}-x_{i+1 j})} + \frac{(\alpha_i + \alpha_{n-j+1})e^{x_{i+1 j}}}{e^{x_{i j+1}} + e^{x_{i+1 j}}}\right).
\end{equation*}
We can observe that $(d^*, \nabla V^*)$ simplifies to equal the negative 
of this: (i) the terms in $(d^*, \nabla V^*)$ that do not involve any $\alpha$ 
parameters cancel; (ii) the terms involving a single $\alpha$ parameter are equal to
\begin{multline*}
-\frac{2(\alpha_i + \alpha_{n-j+1})e^{2x_{i j+1} - 
x_{ij}}}{e^{x_{i+1 j}} + e^{x_{i j+1}}} + 
\frac{2(\alpha_i + \alpha_{n-j+1})e^{2x_{i+1 j} - 
x_{ij}}}{e^{x_{i+1 j}} + e^{x_{i j+1}}} + 2\alpha_{n-j+1} e^{-(x_{ij} - x_{i j+1})} - 2\alpha_{i}e^{-(x_{ij}-x_{i+1 j})} \\
= -2\alpha_{i} e^{-(x_{ij}-x_{i j+1})} + 2\alpha_{n-j+1} e^{-(x_{ij}-x_{i j+1})}  
\end{multline*}
and (iii) the terms involving a product of $\alpha$ parameters are equal to 
\begin{multline*}
 \frac{(\alpha_i + \alpha_{n-j+1})^2e^{2x_{i j+1}}}{(e^{x_{i j+1}} + e^{x_{i+1 j}})^2} -
    \frac{(\alpha_i + \alpha_{n-j+1})^2e^{2x_{i+1 j}}}{(e^{x_{i j+1}} + e^{x_{i+1 j}})^2}
    +\alpha_i^2 - \alpha_{n-j+1}^2 \\
  =  \frac{2\alpha_i(\alpha_i+\alpha_{n-j+1})e^{x_{i j+1}} - 2\alpha_{n-j+1}(\alpha_i + \alpha_{n-j+1})e^{x_{i+1 j}}}{e^{x_{i+1 j}}+e^{x_{i j+1}}}.
\end{multline*}
Therefore \eqref{drift_potential_identity} holds and part (ii) of the Lemma follows by induction. 
\end{proof}

\begin{proof}[Proof of Theorem \ref{finite_tmp_invariant_array}]
Let $S$ be a subset with a boundary given by a down-right path in the orientation of Figure \ref{update_figure}. 
Lemma \ref{gradient_diffusion_lemma} shows that $\frac{1}{2}(\mathcal{L}_S^* + \mathcal{A}_S^*) = 0$ and
Lemma \ref{drift_potential_orthogonality} shows that
\begin{equation*}
 \frac{1}{2}(\mathcal{L}_S^* - \mathcal{A}_S^*) = \frac{1}{2} (\nabla \cdot d_S + (d_S, \nabla V_S)) \pi_S = 0.   
\end{equation*}
As a result $L^*_S \pi_S = 0$ and Lemma \ref{multidimensional_diffusions} proves that  $\pi_S$ is the invariant 
measure for the process with generator $\mathcal{L}_S$. In particular, the case $S = \{(i, j) : i+ j \leq n+1\}$ 
proves the Theorem.
\end{proof}

\begin{proof}[Proof of Theorem \ref{finite_temp_thm}]
A consequence of Theorem \ref{finite_tmp_invariant_array} is that 
\begin{equation*}
 \int_0^{\infty} Z_n(s) \stackrel{d}{=} Y_{n}^* \stackrel{d}{=} 2e^{ X_{11}^*} \stackrel{d}{=} 2 \zeta(1, 1)
\end{equation*}
where $Y_n^*$ is equal in distribution to
$\int_0^{\infty} Z_n(s)$ by the time reversal at the start of this section and by definition $\xi(1, 1) = \log \zeta(1, 1)$.  
The definition of $Z_n$ has $\alpha_1, \ldots, \alpha_n$ in a reversed order 
to the left hand side of Theorem \ref{finite_temp_thm}, however, the distribution of $\zeta(1,1)$
is invariant under reversing the 
order of the parameters --- this follows from the deterministic fact that $\zeta(1, 1)$ 
takes the same value when constructed from the data $\{W_{ij}:i+j \leq n+1\}$ and 
the reflected data $\{W_{ji}:i+j \leq n+1\}$ (in fact the distribution of $\zeta(1,1)$
is left invariant under any permutation of the $\alpha$ parameters as a consequence 
of the same invariance for the process $Z_n$, proven in \cite{o_connell2012}).  
\end{proof}


\subsection{Time reversals and Intertwinings}
The generator $\mathcal{L}$ in \eqref{defnXarray} depends on a sequence of parameters $(\alpha_1, \ldots, \alpha_n)$ 
and we use the notation $(X_{ij}^{(\alpha_1, \ldots, \alpha_n)}(t))_{t \in \mathbb{R}, i + j \leq n+1}$ 
for the process with this generator when we want to make the dependence on the 
$\alpha$ parameters explicit.

%

\begin{proposition}
\label{time_symmetry}
Let $(X_{ij}^{(\alpha_1, \ldots, \alpha_n)}(t) : i + j \leq n +1, t \in \mathbb{R})$ denote the diffusion process 
with generator (\ref{defnXarray}) in stationarity. 
This process has the following properties:
 \begin{enumerate}[(i)]
  \item Time symmetry, \begin{equation}
\label{time_reversal}
 (X_{ij}^{(\alpha_1, \ldots, \alpha_n)}(t))_{t \in \mathbb{R}, i + j \leq n+1} 
 \stackrel{d}{=} (X_{ji}^{(\alpha_n, \ldots, \alpha_1)}(-t))_{t \in \mathbb{R},
 i+j \leq n+1}. 
\end{equation}
\item The marginal distribution of any 
row $(X_{i, n- i+1}, \ldots, X_{i,1})$ run forwards in 
time is a system of exponentially reflecting 
Brownian motions with a wall at the origin with drift vector 
$(-\alpha_{i}, \ldots, -\alpha_n)$.
The marginal distribution of any column 
$(X_{n-j+1, j}, \ldots, X_{1, j})$ 
run backwards in time is a system of exponentially reflecting Brownian motions
with a wall 
at the origin and drift vector $(-\alpha_{n-j+1}, \ldots, -\alpha_1)$.
\end{enumerate}
\end{proposition}
In particular, for equal drifts, part (i) proves that the top particle has the same 
 distribution when run started from its invariant measure either 
 forward or backwards in time: $(X_{11}(t))_{t \in \mathbb{R}} \stackrel{d}{=} (X_{11}(-t))_{t \in \mathbb{R}}$.
This fact does not strike us a priori because the SDEs (\ref{defnX}, \ref{defnX_2}) do not appear to define a reversible diffusion  
unless $n = 1$.


\begin{proof}
The reversed time dynamics of the
 process started in its invariant measure is a Markov process with 
 generator $\hat{\mathcal{L}}$ given by the Doob h-transform of the adjoint 
 generator with respect to its invariant measure, in particular, 
 $\hat{\mathcal{L}}f = \frac{1}{\pi}
 \mathcal{L}^*(\pi f)$. 
 Let $b$ be the drift of the process with generator $\mathcal{L}$ and $a$ the drift of the process 
with generator $\mathcal{A}$ (where we define $\mathcal{A} = \mathcal{A}_S$ when $S = \{(i, j) : i + j \leq n+1\}$).
 The Doob h-transform simplifies due to the fact that $\mathcal{L}^* \pi = 0$ 
 and we obtain 
\begin{equation*}
 \hat{\mathcal{L}}= \frac{1}{2}\Delta + \left(-b -\nabla V\right) \cdot \nabla =  \frac{1}{2} \Delta + a \cdot \nabla 
\end{equation*}
where we use that $-\nabla V = a + b$ from Lemma \ref{gradient_diffusion_lemma}. 
Therefore the time reversal of the process with generator $\mathcal{L}$ 
is the process with generator $\mathcal{A}$.  
The process with generator $\mathcal{A}$ is represented by Figure \ref{update_figure} where
the direction of every interactions is reversed. This 
is equivalent to swapping the $ij$-th particle with the 
$ji$-th particle and reversing the order of the parameters.
This proves part (i).
 
We first prove part (ii) for the columns of the $X$ array. 
When run forwards in time the $X$ array has a nested structure in which particles 
do not depend on particles to the right of them. This means that when 
considering a particular column, say $(X_{n-k+1 k}, \ldots, X_{1k})$, 
we can restrict to a subarray 
$(X_{ij} : j \geq k, i+j \leq n+1)$ where this is 
the rightmost column. 
The top row of this subarray run forwards in time is 
a system of exponentially reflecting Brownian motions with a wall
with drift vector $(-\alpha_1, \ldots, -\alpha_{n-k+1})$.
Combining this with 
the time reversal in part (i) proves that the column $(X_{n-k+1 k}, \ldots, X_{1k})$ run backwards in time 
is a system of exponentially reflecting Brownian motions with a wall
with drift vector $(-\alpha_{n-k+1}, \ldots, -\alpha_{1})$. This proves the result for every column in the $X$ array. 
The result for rows then follows from the time reversal in part (i). 
%
%
\end{proof}
This easily extends to show that the time reversal of the process with generator $\mathcal{L}_S$ when run in its invariant 
measure $\pi_S$ is 
the process with generator $\mathcal{A}_S$ for any subset $S$ with a down-right boundary. 

Let $Q^n_t$ denote the transition semigroup for $n$ exponentially reflecting Brownian motions with a wall. 
Considering the process $(X_{ij} : i + j \leq n+1)$ run in stationarity 
leads to an intertwining between $Q^{n-1}_t$ and $Q^{n}_t$. 
The intertwining kernel is given by the transition kernel of a Markov chain  constructed from the  point-to-line log-gamma polymer as follows.
The log partition functions form a Markov chain $(\mathbf{\xi}_k)_{ 1\leq k \leq n}$ 
 where $\mathbf{\xi}_k = (\xi(k, n-k+1), \ldots, \xi(k, 1))$.  The Markov property  for this chain follows from the
local update rule for partition functions 
$\zeta_{ij} = (\zeta_{i j+1} + \zeta_{i+1 j}) W_{ij}$ and
equivalently for the log partition functions
$\xi_{ij} = \log W_{ij} + \log (e^{\xi_{i j+1}} + e^{\xi_{i+1 j}})$. 
 We let $P_{k-1 \rightarrow k}$ denote the 
 transition kernel  for this chain.

We start the process $(X_{ij})_{t \in \mathbb{R}, i+j \leq n+1}$ 
in stationarity and consider 
two different ways of calculating the probability density function 
of the vector 
\begin{equation}
\label{intertwining_vector}
P (X_{n-1, 2}(0) \in dx_{n-1, 2}, \ldots, X_{1, 2}(0) \in dx_{1 2}, X_{n, 1}(t) \in dz_{n1}, \ldots, X_{1, 1}(t) \in dz_{11}).
\end{equation}
Let $\mathbf{x}_2 = (x_{n-1, 2}, \ldots, x_{12})$ and let $\mathbf{z}_1 = (z_{n1}, \ldots, z_{11})$.
\begin{enumerate}[(i)]
 \item Calculate \eqref{intertwining_vector} by integrating over $X_{n, 1}(0), \ldots, X_{1, 1}(0)$  
as an intermediate step. When run forwards in time, the evolution of the top row of the $X$ array is independent
of the rest of the array due to the direction of interactions. Therefore $(X_{n-1, 2}(0), \ldots, X_{1, 2}(0)$ 
and $X_{n, 1}(t), \ldots, X_{1, 1}(t))$ are conditionally independent given  $X_{n, 1}(0), \ldots, X_{1, 1}(0)$. 
Letting $\mathbf{x}_1 = (x_{n1}, \ldots, x_{11})$ the probability density \eqref{intertwining_vector} equals
\begin{equation}
\label{intertwining_calculation1}
\int P_{n-1 \rightarrow n}(\mathbf{x}_2, 
 \mathbf{x}_{1})
  Q^n_t(\mathbf{x}_{1}, \mathbf{z}_{1}) d\mathbf{x}_{1} 
\end{equation}
\item Calculate \eqref{intertwining_vector} by integrating over $X_{n-1, 2}(t), \ldots, X_{1, 2}(t)$ as an intermediate step.
When run backwards in time, the evolution of the second row is not affected by the top row 
of the $X$ array. Therefore $(X_{n-1, 2}(0), \ldots, X_{1, 2}(0)$ 
and $X_{n, 1}(t), \ldots, X_{1, 1}(t))$ are conditionally independent given $X_{n-1, 2}(t), \ldots, X_{1, 2}(t)$.
Letting $\mathbf{z}_2 = (z_{n-1, 2}, \ldots, z_{12})$ the probability density \eqref{intertwining_vector} equals
\begin{equation}
\label{intertwining_calculation2}
\int   Q^{n-1}_t(\mathbf{x}_{2}, \mathbf{z}_{2}) P_{n-1 \rightarrow n}(\mathbf{z}_2, 
 \mathbf{z}_{1})
 d\mathbf{z}_{2} 
\end{equation}
\end{enumerate}

The equality of \eqref{intertwining_calculation1} and \eqref{intertwining_calculation2} proves an intertwining
between $Q_t^{n-1}$ and $Q_t^n$ with intertwining kernel $P_{n-1 \rightarrow n}$. 
This can be expressed in operator notation as
\begin{equation*}
 Q_t^{n-1} P_{n-1 \rightarrow n} = P_{n-1 \rightarrow n} 
 Q_t^{n}.  
\end{equation*}

\subsection{Zero-temperature limits}
We can take a zero temperature limit of the construction we have considered above. 
In the limit, particles follow the coupled system of SDEs: for $j = 1, \ldots, n$,
\begin{equation*}
 dX_{1j}(t) = dB_{1j}(t) - \alpha_{n-j+1} dt + dL_{1j}^1(t)  
 \end{equation*}
 and for $i > 1$ and $ i + j \leq n+1$, 
 \begin{equation*}
 dX_{ij}(t) = dB_{ij}(t) - \alpha_{n - j+1} 1_{\{X_{ij} < X_{i-1, j+1}\}}dt 
 + \alpha_{i-1} 
 1_{\{X_{ij} > X_{i-1, j+1}\}}dt + dL_{i j}^1(t) - dL_{ij}^2(t) 
\end{equation*}
where (i) $L_{ij}^1$ is the local time process at zero of 
       $X_{ij}- X_{i, j-1}$ for $i+j < n+1$, (ii) $L_{ij}^1$ is the local time process at zero of 
$X_{ij}$ for $i+j = n$, and (iii) $L_{ij}^2$ is the local time process at zero of 
$X_{ij} - X_{i-1, j}$
for $i \geq 2$.
This process can be represented by Figure \ref{fig_Xarray} where the 
interaction $\rightarrow$ is now reflection and the interaction
$\leadsto$ is now a weighted indicator function.
The zero-temperature limit of the field of 
log partition functions is the field of point-to-line last passage percolation times 
$\{G(i, j): i + j \leq n+1\}$ (see \cite{bisi_thesis, bisi_zygouras}) and 
it is natural to expect that $\{G(i, j): i + j \leq n+1\}$ is the invariant measure 
of $\{X_{i j}: i + j \leq n+1\}$. However, we do not prove this because the 
discontinuities in the drifts means that the conditions for 
Lemma \ref{multidimensional_diffusions} are no longer satisfied. 
Instead, we argue that a second proof of Theorem \ref{equality_law} can be provided as
a zero temperature limit 
of Theorem 
\ref{finite_tmp_invariant_array}. We can introduce an extra inverse 
temperature parameter $\beta$ 
into the definitions of the processes $X, Y$ and $Z$ given in this section and the 
results of this Section continue to hold. In particular,
Theorem \ref{finite_tmp_invariant_array} 
and the time reversal in Section \ref{finite_temp_time_reversal} establish that
\begin{equation*}
\frac{1}{\beta} \log  \int_{0 = s_0 < s_1 \ldots < s_n < \infty} e^{\beta \sum_{i=1}^n B_i^{(-\alpha_i)}(s_i) - B_i^{(-\alpha_i)}(s_{i-1})}
 ds_1 \ldots ds_n \stackrel{d}{=} \frac{1}{\beta} \log 2 \sum_{\pi \in \Pi_n^{\text{flat}}} \prod_{(i, j) \in \pi} W_{ij}^{(\beta)}
\end{equation*}
where $\{W_{ij}^{(\beta)} : i + j \leq n+1\}$ are random variables with inverse gamma distributions and rates 
$\beta^{-1}(\alpha_i + \beta_j)$.
As $\beta \rightarrow \infty$, the left hand side
converges almost surely by Laplace's Theorem 
and the right hand side 
converges by \cite{bisi_thesis, bisi_zygouras} to give,
\begin{equation*}
 \sup_{0 = s_0 \leq \ldots \leq s_n < \infty} \sum_{i=1}^n B_i^{(-\alpha_i)}(s_i) - B_i^{(-\alpha_i)}(s_{i-1})
 \stackrel{d}{=} \max_{\pi \in \Pi_n^{\text{flat}}} \sum_{(i,j) \in \pi} W_{ij}.
\end{equation*}
The time reversal in Proposition \ref{time_change_refl} allows the distribution of the left hand side to be 
identified as $Y_n^*$. This argument is easily extended to prove Theorem \ref{equality_law} in its entirety.

\section{Further random matrix interpretations}
\label{symmetric_lue_section}

We now discuss an alternative version of Theorem \ref{sup_lpp} that 
connects two families of random matrices. Let $X$ be a 
symmetric complex matrix of size $n \times n$ 
where for $i < j$ the entries $X_{ij}$ are independent complex Gaussian with  
mean zero and variance given by $\frac{1}{2(\alpha_i + \alpha_j)}$ and the entries
along the diagonal $X_{ii}$ are independent complex Gaussian with mean 
zero and variance $\frac{1}{2\alpha_i}$. 
We call the matrix $X^* X$ a 
perturbed symmetric LUE matrix. 
In the case when the $\alpha_i$ are distinct, we will show
the eigenvalues of $X^* X$ have a density with respect to Lebesgue measure given by 
 \begin{equation}
 \label{eigenvalue_law}
 f(\lambda_1, \ldots, \lambda_n) = \frac{ \prod_{i = 1}^n  \alpha_i 
 \prod_{i < j}(\alpha_i + \alpha_j)}{\prod_{i < j} (\alpha_i - \alpha_j)} 
 \text{det}(e^{-\alpha_i \lambda_j})_{i, j =1}^n.
 \end{equation}
When some of the $\alpha_i$ coincide this can be evaluated as a limit and in
the case when all $\alpha_i$ are equal it agrees with the eigenvalue 
density of LOE.
Our interest in this random matrix ensemble arises from the connection of its 
eigenvalue density to point-to-line last passage percolation.
In the case when 
the parameters are equal, a similar case
appears in Theorem 7.7 of \cite{baik_rains} but with a different variance along the diagonal for 
the random matrix model and different rates along the diagonal for the exponential data -- that the variances and rates 
along the diagonal can be tuned is a property of RSK (for example, see Chapter 10 of \cite{forrester_book})
and that the sum of diagonal 
entries is the trace of a matrix.
Point-to-point last passage percolation with inhomogeneous rates for the 
exponential data was related to random matrices 
with inhomogeneous variances in \cite{borodin_peche, dieker_warren}.

To calculate the eigenvalue density we compute the Jacobian (see Chapter 1 of 
\cite{forrester_book} for related examples),
\begin{equation*}
\label{eigenvalues_Jacobian}
dX \propto
\prod_{j < k} \lvert \lambda_k - \lambda_j \rvert \prod_j d\lambda_j d\Omega
\end{equation*}
of the transformation from matrix elements $X$ to the eigenvalues $\lambda$ 
and angular variables $\Omega$. 
The choice of parameters ensures the distribution on matrices 
can be expressed as a trace,
\begin{equation*}
 P(X) = c_n \prod_{i = 1}^n \alpha_i 
 \prod_{i < j}(\alpha_i + \alpha_j)\exp\left(-\sum_{i=1}^n 
 \alpha_i \lvert x_{ii} \rvert^2 - \sum_{i < j} 
 (\alpha_{i}
 + \alpha_j) \lvert x_{ij} \rvert^2 \right) d\mathbf{x} 
 \propto \exp\left(-\text{Tr}(A X^* X) \right) d\mathbf{x}
\end{equation*}
where $d\mathbf{x}$ is Lebesgue measure on the independent (complex)
entries $(x_{ij} : i \leq j)$ of the matrix $X$,
the matrix $A = \text{diag}(\alpha_1, \ldots \alpha_n)$ and 
$c_n$ is a constant.
Let the singular value decomposition be given by $X = U D U^T$ where 
$U \in \mathbb{U}(n)$ the set of $n \times n$ unitary matrices, 
$D = \text{diag}(\sqrt{x_1}, \ldots, \sqrt{x_n})$
is the diagonal matrix consisting of the singular values of $X$
and the singular
value decomposition takes this form due to the symmetry of $X$ (also referred to as 
the Autonne-Takagi factorisation). Let 
$V = U^T \in \mathbb{U}(n)$ and $\Lambda = D^2 = \text{diag}(x_1, \ldots, x_n)$. 
The joint density of eigenvalues is given by 
\begin{equation*}
 f(\lambda_1, \ldots, \lambda_n) = \int_{V \in \mathbb{U}(n)} e^{-\text{Tr}(A V \Lambda V^*)} 
 \Delta(x) dV = \frac{ \prod_{i = 1}^n  \alpha_i 
 \prod_{i < j}(\alpha_i + \alpha_j)}{\prod_{i < j} (\alpha_i - \alpha_j)} 
 \text{det}(e^{-\alpha_i \lambda_j})_{i, j = 1}^n
\end{equation*}
where the integral over the unitary group is calculated by the 
Harish-Chandra-Itzykson-Zuber formula.

This agrees with the density of the output of RSK 
when applied to last passage percolation with symmetric exponential data with modified 
rates along the diagonal as described in Section \ref{nr_br}. 
Therefore 
we obtain the following extension of Theorem \ref{sup_lpp}:
\begin{proposition}
\label{symmetric_lue}
Let $\xi_{\text{max}}$ 
 denote the largest eigenvalue of a perturbed symmetric LUE matrix 
 with parameters $\alpha_i$, let $(H(t):t \geq 0)$ be an $n \times n$ 
 Hermitian Brownian motion, let $D$ be an $n \times n$ diagonal matrix 
 with diagonal entries $\alpha_j > 0$ for each $j = 1, \ldots, n$
 and let $e_{ij}$ be 
 an independent collection of 
exponential random variables indexed by the lattice $\mathbb{N}^2$
with rate $\alpha_i + \alpha_{n + 1 -j}$.
Then 
 \begin{equation*}
  2\sup_{t \geq 0} \lambda_{\max}(H(t) - t D) \stackrel{d}{=} 
  2 \max_{\pi \in \Pi_n^{\text{flat}}}
  \sum_{(i j) \in \pi} e_{ij} \stackrel{d}{=}  \xi_{\max}.
 \end{equation*}
\end{proposition}

There does not appear to be any 
process level equality between a vector of last passage percolation times and 
the largest eigenvalues of minors 
of either (i) the perturbed symmetric LUE or (ii) 
the Laguerre orthogonal ensemble (nor does the connection between last passage 
percolation and LOE generalise to 
non-equal rates).

\section{Distribution of the largest particle}
\label{toda_lattice_section}

In this section we consider 
the 
distribution of the largest particle of
the system of reflected Brownian motions with a wall in its invariant
measure. This has a number of alternative representations from 
Theorem \ref{equality_law}, Proposition \ref{time_change_refl} 
and Proposition \ref{symmetric_lue} in particular as a point-to-line 
last passage percolation time.
A variety of expressions have been found for this
in \cite{baik_rains, bisi_zygouras, bfps, forrester2004, johnstone} which are 
convenient for asymptotic analysis.
The expression that arises most naturally from  
Theorem \ref{invariant_measure} is an expression in terms of the $\tau$-function 
of a Toda lattice given
in Forrester and Witte, Section 5.4 of \cite{forrester2004} 
(also see Proposition 10.8.1 of Forrester \cite{forrester_book}). Their result is
part of a more general and powerful theory developed in a series of papers 
(see \cite{forrester2004} and the references within); however, it is natural
to see how expressions in terms of a Toda lattice arise from Theorem \ref{invariant_measure}
in an elementary manner.

\begin{proposition}
\label{largest_particle}
Let $F(x) = P(Y_n^* \leq x) = P(G(n, n) \leq x)$. 
 \begin{enumerate}[(i)]
  \item When the drifts are equal $\alpha_1 = \ldots = \alpha_n$, this is given by a Wronskian
  \begin{IEEEeqnarray*}{rCl}
 F(x)   = \text{det}(f_{i-1}^{(j-2)}(x))_{i, j = 1}^n 
\end{IEEEeqnarray*}
where the functions $f_i^{(j)}$ are defined in equation (\ref{f_defn})
and $f^{(-1)}(x) = \int_0^x f(u) du$.
Furthermore, this is the $\tau$-function for a Toda lattice 
equation, 
\begin{IEEEeqnarray*}{rCl}
F(x)   
& = & \frac{1}{Z} e^{-nx} x^{-n^2/2 + n/2}
\text{det}\left(\left(
x \frac{d}{dx}\right)^{i+j-2}  \sqrt{\frac{2}{\pi}} \sinh(x) \right)_{i, j = 1}^n 
\end{IEEEeqnarray*}
where $Z$ is a normalisation constant.
\item When the drifts are distinct, 
 \begin{equation*}
  F(x) = e^{-\sum_{i=1}^n \alpha_i x} \text{det}\left(
  \begin{matrix}
   f_1(x) & D^{\alpha_1} f_1(x) & \ldots & D^{\alpha_1 \ldots \alpha_{n-1}} f_1(x) \\
   f_2(x) & D^{\alpha_1} f_2(x) & \ldots & D^{\alpha_1 \ldots \alpha_{n-1}} f_2(x) \\
   \vdots & \vdots & \ddots & \vdots \\
   f_n(x) & D^{\alpha_1} f_n(x) & \ldots & D^{\alpha_1 \ldots \alpha_{n-1}} f_n(x) 
  \end{matrix}\right)_{i, j = 1}^n
 \end{equation*}
 where $f_i(x) = e^{\alpha_i x} - e^{-\alpha_i x}$.
 \end{enumerate}
\end{proposition}
For the interpretation in terms of the Toda lattice equation we 
let $g[n](x) =  \text{det}((x \frac{d}{dx})^{i+j-2} 
 \sqrt{\frac{2}{\pi}} \sinh(x))_{i, j = 1}^n$ and observe that $g$ 
 solves the Toda lattice equation, 
\begin{equation*}
 \left(x \frac{d}{dx}\right)^2 \log g[n] = \frac{g[n+1]g[n-1]}{g[n]^2}
\end{equation*}
with $g[0] = 1$ and 
$g[1](x) = \sqrt{\frac{2}{\pi}} \sinh(x)$.
The Toda lattice equation is often expressed in terms of
$I_{1/2}$ the modified Bessel function of the 
first kind by $I_{1/2}(x) = (\sqrt{2/\pi x}) \sinh(x).$

%
%

\begin{proof}[Proof of Proposition \ref{largest_particle}]
In the homogeneous case we obtain from Theorem \ref{invariant_measure} that
\begin{equation*}
P(Y_n^* \leq x)  =   P(Y_1^*, \ldots, Y_n^* \leq x) 
= \int_{x_1 \leq \ldots x_n \leq x} \text{det}(f_{i-1}^{(j-1)}(x_j))_{i, j = 1}^n 
dx_1 \ldots dx_n.\\
\end{equation*}
We perform the integral in $x_n$ which leads to an integrand given by 
a determinant where the last column in the determinant above has been 
replaced by $f_{i-1}^{(n-2)}(x) - 
f_{i-1}^{(n-2)}(x_{n-1})$. The second term can be removed 
from the last column by column 
operations. This procedure, of integration and column operations, 
can be applied iteratively to 
the variables $x_{n-1}, \ldots, x_1$ and leads to the required formula.
In the inhomogeneous case we apply the same steps: 
in particular, we obtain from Theorem \ref{invariant_measure} that
\begin{equation*}
 P(Y_n^* \leq x)  =   P(Y_1^*, \ldots, Y_n^* \leq x) 
 = \int_{x_1 \leq \ldots \leq x_n \leq x} e^{-\sum_{i=1}^n \alpha_i x_i} \text{det}(D_{x_j}^{\alpha_1 \ldots \alpha_j}
  f_i(x_j))_{i, j = 1}^n dx_1 \ldots dx_n.
\end{equation*}
We perform the integral in $x_n$ which replaces the last column of the determinant 
by
$e^{-\alpha_n x} D^{\alpha_1 \ldots \alpha_{n-1}} f_i(x) 
  - e^{-\alpha_n x_{n-1}} D^{\alpha_1 \ldots \alpha_{n-1}} f_i(x_{n-1}).$
The second term can be removed from the last column by 
column operations and the results follows by iteratively applying this procedure in 
the variables $x_{n-1}, \ldots, x_1$.

We now show the second expression in part (i) is equal to the first expression in (i) by 
a series of row and column operations.
We observe that applying a series of column operations shows that 
\begin{equation}
\label{toda_eq1}
e^{-nx} x^{-n^2/2 + n/2} \text{det}\left(\left(x 
\frac{d}{dx}\right)^{i+j-2} \sqrt{x} I_{1/2}(x)
 \right)_{i, j = 1}^n      = \text{det}\left(
\frac{d^{j-1}}{dx^{j-1} }\left(\left(\left(x \frac{d}{dx}\right)^{i-1} 
\sqrt{x} I_{1/2}(x)\right) e^{-x}\right)\right)_{i, j = 1}^n
\end{equation}
where we can apply column operations to the left hand side in order to obtain
that the application of $(x \frac{d}{dx})^{j-1}$
in the $j$-th column is equivalent to the application of 
$x^{j-1} \frac{d^{j-1}}{dx^{j-1}},$ and after this observation, the $x^{j-1}$ term in each 
column can be brought 
outside of the determinant to cancel the polynomial prefactor. The exponential prefactor on the left hand side can be brought inside the 
determinant and, using column operations, inside the derivative operators 
$\frac{d^{j-1}}{dx^{j-1}}$.

We prove by induction on $i$ that we can add on multiples of rows $(1, \ldots, i-1)$ to  
the $i$-th row of the matrix on the right hand side of (\ref{toda_eq1})
to obtain equality with the
the matrix $(f_{i-1}^{(j-2)}(x))_{i, j = 1}^n.$
We only need to check this for the entry in the first column since 
both sides of (\ref{toda_eq1}) 
share the same derivative structure in columns. 
We observe that equality holds (without any row operations)
for the first row: $\sqrt{x} I_{1/2}(x) e^{-x} = f_0^{(-1)}(x)$. 
Assuming the inductive hypothesis, for each $i \geq 0$ the entry 
in the $(i+2)$-nd row and $2$-nd column on the right hand side of (\ref{toda_eq1}) 
is given by $ x f_i(x) + xf_i'(x) + f_i(x) + \int_0^x f_i(u) du$ 
by using the relationships between the entries of the matrix -- in particular, we 
assume the entry in the $(i+1)$-st row and $2$-nd column is given by $f_i$; then
integrate to find the entry in the $(i+1)$-st row and first column; 
we then  
find the entry in the $(i+2)$-nd row and first column 
as $e^{-x} x\frac{d}{dx}(e^{x} f_i^{(-1)}(x)) = xf_i(x) + xf_i^{(-1)}(x)$, and
differentiate to 
find the entry in the $(i+2)$-nd row and $2$-nd column stated above.
To simplify this expression, we prove the following identity: there exist constants $c_1, \ldots, c_{i}$ 
such that
\begin{equation}
\label{identity}
 x f_i(x) + xf_i'(x) + \int_0^x f_i(u) du = ( i+1) f_{i+1}(x) + c_i f_i(x) + \ldots 
 + c_1 f_1(x)
\end{equation}
which shows that after applying row operations the matrix will be in the required 
form (the factor of $(i+1)$ can be absorbed into the normalisation constant). 
We note that the function $f_0$ is \emph{not} used on the right hand side.
After applying these row operations 
the entry in the $(i+2)$-nd row and $1$-st column will be given by 
$f^{(-1)}_{i+1}$ by using an additional boundary condition:
that the entries in the first column of the matrix on the 
right hand side of (\ref{toda_eq1}) are all zero at zero.
We prove equation (\ref{identity}) by induction and let 
\begin{equation*}
h_{i+1}(x)  = x f_i(x) + xf_i'(x) + \int_0^x f_i(u) du
\end{equation*}
For the base case of the identity, observe that 
$f_1(x) = xf_0(x) + xf_0'(x) + \int_0^x f_0(u) du$ from  
$f_0(x) = e^{-2x}$ and an explicit expression for $f_1(x) = -xe^{-2x} + \frac{1}{2} 
- \frac{1}{2} e^{-2x}$. 
For the inductive step, observe that 
\begin{equation*}
\mathcal{G}^* h_{i+1}(x) = x \mathcal{G}^* f_i(x) + x \mathcal{G^*} f_i'(x) + \mathcal{G^*} \int_0^x f_i(u) du + 
f_i''(x) + 2f_i'(x) + f_i(x) 
= h_i(x) + f_i(x) + 2f_{i-1}(x)  
\end{equation*}
where the second equality follows by using the defining property 
of the $f_i$, namely that $\mathcal{G}^* f_{i} = f_{i-1}$, and $\mathcal{G^*} \int_0^x f_i(u) du = \int_0^x f_{i-1}(u) du$
by an additional boundary condition that both sides are zero at zero.
The inductive hypothesis 
means there exists constants such that $h_i = if_{i}
+ \tilde{c}_{i-1} f_{i-1} + \ldots + \tilde{c}_1 f_1$. Therefore 
$\mathcal{G}^* h_{i+1}$ can 
be expressed in terms of the functions $f_1, \ldots, f_{i}$, and we can 
choose the constants $c_i, \ldots, c_1$ in equation (\ref{identity}) such that the
operator $\mathcal{G}^*$ applied to the right 
hand side of (\ref{identity}) agrees with $\mathcal{G}^* h_{i+1}$.  
The boundary conditions $h_{i+1}(0) = h_{i+1}'(0) = 0$ also
agree with the right hand side of equation (\ref{identity}). Therefore this completes 
the proof of the identity and in turn this identity then proves the Proposition.
%
\end{proof}

\paragraph{Acknowledgements.} WF is supported by EPSRC as part of the MASDOC DTC, Grant No. EP/HO23364/1.

\nocite{*}

 \bibliographystyle{abbrv}
 \bibliography{loe_bib}

\end{document}